\documentclass[11pt]{article}

\usepackage[tbtags]{amsmath}
\usepackage{amssymb}
\usepackage{amsthm}
\usepackage[misc]{ifsym}
\usepackage{cases}
\usepackage{mathrsfs}
\usepackage{bm}
\usepackage{bbm}
\usepackage{color}
\usepackage{amsfonts}
\usepackage{stmaryrd}
\usepackage[active]{srcltx}
\usepackage{enumerate}
\usepackage{mathtools}
\usepackage{empheq}
\usepackage{wasysym}
\usepackage{verbatim}
\usepackage{graphicx}
\usepackage[
bookmarks=true,         %generates bookmarks for all entries in the "Table of Contents"
bookmarksnumbered=true, %includes chapter-/header-/section-/subsection-/... -
colorlinks=true, pdfstartview=FitV, linkcolor=blue, citecolor=blue,
urlcolor=blue]{hyperref}
\usepackage{microtype}

\numberwithin{equation}{section}
\normalsize

\title{\bf A Partially Observed Stochastic Linear Stackelberg Differential Game with Poisson Jumps under Mean-Variance Criteria
\thanks{This work is financially supported by the National Natural Science Foundations of China (12471419, 12271304), and Shandong Provincial Natural Science Foundation (ZR2024ZD35).}}
\author{\normalsize Jingtao Lin\thanks{\it School of Mathematics, Shandong University, Jinan 250100, P.R. China, E-mail: linjingtao@mail.sdu.edu.cn},\quad Jingtao Shi\thanks{\it Corresponding author. School of Mathematics, Shandong University, Jinan 250100, P.R. China, E-mail: shijingtao@sdu.edu.cn}}

\newtheorem{proposition}{Proposition}[section]
\newtheorem{theorem}{Theorem}[section]

\newtheorem{lemma}{Lemma}[section]
\newtheorem{remark}{Remark}[section]
\newtheorem{assumption}{Assumption}[]

\begin{document}

\maketitle

\noindent{\bf Abstract:}\quad In this paper, a partially observed stochastic linear Stackelberg differential game with mean-variance criteria is studied. Randomness comes from Brownian motions and Poisson random measures. which leads to a circular dependency. We follow the orthogonal decomposition method to overcome the circular dependency of the control and state processes. Both original problems of the follower and leader are decomposed into several fully observed problems with mean-variance criteria. During these processes, non-linear stochastic filtering with Poisson random measures, developed in this paper, plays an important role. Besides the follower's problem is embedded into a class of auxiliary stochastic linear-quadratic optimal control problem of stochastic differential equations with Poisson jumps, the leader's problem is also embedded into a class of auxiliary stochastic linear-quadratic optimal control problem of forward-backward stochastic differential equations with Poisson jumps. Observable state feedback Stackelberg equilibria are obtained, via some Riccati equations.

\vspace{2mm}


\vspace{2mm}


\section{Introduction}\label{Sec Intro}

In recent years, the Stackelberg (also known as the leader-follower) game has become an active topic in the research of non-zero-sum games. The Stackelberg solution of the game is obtained when one player is forced to wait for the other player to announce his decision before making his own decision. Problems of this nature often arise in economics, where decisions must be made by two parties, and one party is subordinate to the other. As a result, the subordinate party must wait for the decision of the other party before formulating its own decision.

The research of Stackelberg game can be traced back to the pioneering work by von Stackelberg \cite{Stackelberg34} in static competitive economics. Based on other researchers' work, Yong \cite{Yong02} extended the {\it stochastic linear quadratic} (SLQ) Stackelberg differential game to a rather general framework, where the coeffcients could be random matrices, the control variables could enter the diffusion term of the state equation and the control weight matrices in the cost functionals need not to be positive definite. The problem is studied by the {\it stochastic maximum principle} (SMP) approach. The problem of the leader is described as an SLQ optimal control problem of a {\it forward-backward stochastic differential equation} (FBSDE). Moreover, it is shown that the open-loop solution admits a state feedback representation if a new stochastic Riccati equation is solvable.

Unlike the complete information settings assumed in previous literature, certain agents with asymmetrical roles often acquire only partial information. This is typically attributed to multiple factors. Information transmission delays can impede the timely arrival of data, while inequalities in market competition might limit an agent's access to comprehensive information. Additionally, the existence of private information among different classes and government imposed restrictive policies also contribute to this partial information situation. Stochastic Stackelberg differential games with partial and asymmetric information were studied in \cite{SWX16,SWX20,ZS22}.

Recently, without imposing additional requirements on the admissible control, Sun and Xiong \cite{SX23} studied an SLQ optimal control problem. Using the {\it orthogonal decomposition technique} of state process, with the help of linear filtering theory, \cite{SX23} devided the original cost functional into two sub-cost functionals. The first one is dependent on the control process and the filtering process, which is completely observed. The second one is independent of the choice of the admissible control. The orthogonal decomposition technique was extendedly used to the mean-field type stochastic systems \cite{MB24} and Stackelberg games \cite{LSZ24}.

Random jumps are often used to model sudden major events in financial markets. Lin \cite{Lin21} investigated the SLQ Stackelberg differential game with Poisson jumps. However, in their setting, the coefficients involved in the state equation and cost functional are deterministic matrices, and the control variable does not enter the jump term. Moon \cite{Moon21} generalized \cite{Lin21} to consider the SLQ Stackelberg differential game with Poisson random measure, under random coefficients and allowign the jump term of state equation affected by both follower's and leader's control variables. More recent developments can be founded in \cite{Moon25,WXZ24} and the references therein.

Stochastic Stackelberg differential games have a wide range of applications in the field of mathematical finance, management science and economy. Especially, the optimal reinsurance-investment problem has long been a popular topic in actuarial science. Chen and Shen \cite{CS19} investigated the Stackelberg differential reinsurance games under the mean-variance framework, in which an insurer and a reinsurer are considered as two players of the game. Recently, Huang and Zhu \cite{HZ25} studied the Stackelberg games with asymmetric information, where both investors care not only about their own terminal wealth, but also about its relative performance compared to the average terminal wealth of both palyers. More related works about stochastic Stackelberg differential games with mean-variance costs can be founded in \cite{LY22,GLS24,LL25,LY25,WLZ25}. However, in these works, Stackelberg equilibrium strategies and equilibrium value functions are achieved by {\it Hamilton-Jacobi-Bellman} (HJB) equations.

Motivated by the above work, in this paper, we study a partially observed stochastic linear Stackelberg differential game with mean-variance criteria. Randomness comes from Brownian motions and Poisson random measures, which makes our model closer to reality. Our work distinguishes itself from the existing literatures in the following aspects:

(1) The model and problem are new. A partially observed stochastic linear Stackelberg differential game with mean-variance criteria is considered. Both Brownian motions and Poisson random measures are involved in our setting.

(2) The system is complex and the research method is not usual. Since the state is partially observed, the linear and non-linear stochastic filtering with Poisson random measure is used. Following the orthogonal decomposition technique of \cite{SX23}, the original partially observed problems, for both the follower and leader's, are brought into two independent fully observed ones with mean-variance criteria, respectively, where the corresponding state processes come from stochastic filtering.

(3) To deal with the mean-variance criteria, following the embedding theorem developed in Zhou and Li \cite{ZL00}, we transform the problem into the SLQ framework.

(4) For the leader's problem, which is an SLQ optimal control problem of FBSDEPs, we extend the decoupling method in \cite{HJX23,LSZ24} to the Poisson jump setting.

(5) SMP is applied and state feedback representation of open-loop Stackelberg equilirria is obtain via some Riccati equations.

The rest of this paper is organized as follows. In Section \ref{Sec Problem formulation}, we formulate our partially observed stochastic linear Stackelberg differential game with mean-variance criteria, where the randomness comes from both Brownian motions and Poisson random measures. The follower's and leader's problems are studied in Section \ref{Sec follower's} and Section \ref{Sec leader's}, respectively, and open-loop Stackelberg equilibrium strategies and their state feedback representations are obtained. In Section \ref{Sec Conclu}, we give some concluding remarks. Some results about linear and non-linear stochastic filtering with Poisson random measures, are postponed in the Appendix \ref{Sec Filter}.

\section{Problem formulation}\label{Sec Problem formulation}

Throughout the paper, let $\mathbb{R}^n$ denote the \emph{n}-dimensional Euclidean space with standard Euclidean norm $\vert\cdot\vert $ and standard Euclidean inner product $\left\langle\cdot,\cdot\right\rangle $. The transpose of a vector (or matrix) $x$ is denoted by $\mathbf{\emph{x}}^\top$. $\mbox{Tr}(A)$ denotes the trace of a square matrix $A$. Let $\mathbb{R}^{n\times m}$ be the Hilbert space consisting of all $n\times m$-matrices with the inner product $\left\langle A,B\right\rangle := \mbox{Tr}(AB^\top$) and the Frobenius norm $\vert A \vert:=\langle A,A\rangle^\frac{1}{2}$. Denote the set of symmetric $n\times n$ matrices with real elements by $\mathbb{S}^n$. If $M\in\mathbb{S}^n$ is positive (semi-) definite, we write $M > (\geqslant) 0$.

Let $T>0$ be fixed. Consider a complete filtered probability space $(\Omega,\mathcal{F},\{\mathcal{F}_t\}_{0\leqslant t\leqslant T},\mathbb{P})$ on which two standard independent Brownian motions $W_1\equiv\{W_1(t)\in\mathbb{R};t\geqslant0\}$ and $W_2\equiv\{W_2(t)\in\mathbb{R};t\geqslant0\}$ are defined. Let $(E,\mathcal{B}(E))$ be a Polish space with the $\sigma$-finite measure $\nu_i$ on $E_i$ and $E_i\subset E$, for $i=1,2$. Suppose that $N_i(de,dt)$ are Poisson random measures on $(\mathbb{R}^+\times E_i,\mathcal{B}(\mathbb{R}^+)\times\mathcal{B}(E_i))$ under $\mathbb{P}$ and for any $A\in\mathcal{B}(E_i)$, $\nu_i(A)<\infty$, the compensated Poisson random measure is given by $\tilde{N}_i(de,dt)\coloneqq N_i(de,dt)-\nu_i(de)dt$. Moreover, $W_1,W_2,N_1,N_2$ are mutually independent under $\mathbb{P}$, and let $\mathcal{F}_t^{W_1},\mathcal{F}_t^{W_2},\mathcal{F}_t^{N_1},\mathcal{F}_t^{N_2}$ be the $\mathbb{P}$-completed natural filtrations generated by $W_1,W_2,N_1,N_2$, respectively. Set $\mathcal{F}_t\coloneqq\mathcal{F}_t^{W^1}\vee\mathcal{F}_t^{W_2}\vee\mathcal{F}_t^{N_1}\vee\mathcal{F}_t^{N_2}\vee\mathcal{N}$ and $\mathbb{F}:=\{\mathcal{F}_t\}_{0\leqslant t\leqslant T}$, where $\mathcal{N}$ denotes all $\mathbb{P}$-null sets. $\mathbb{E}$ is the expectation under $\mathbb{P}$.

In our Stackergerg game, there exists two players: one is the leader and the other is the follower. Let $u_i\equiv\{u_i(t)\in\mathbb{R}^m;0\leqslant t\leqslant T\}$, $i=1,2$ be the control processes of the follower and the leader, respectively. The state process $X^{u_1,u_2}\equiv\{X^{u_1,u_2}(t)\in\mathbb{R}^n;0\leqslant t\leqslant T\}$ evolves in line with the following linear {\it stochastic differential equation with Poisson random jumps} (SDEP):
\begin{equation}\label{state equation 2.1}
	\left\{
	\begin{aligned}
		dX^{u_1,u_2}(t)&=\left[A(t)X^{u_1,u_2}(t)+B_1(t)u_1(t)+B_2(t)u_2(t)\right]dt\\
		&\quad +\sum_{i=1}^2C_i(t)dW_i(t)+\sum_{i=1}^2\int_{E_i}D_i(t,e)\tilde{N}_i(de,dt),\\
		X^{u_1,u_2}(0)&=x_0,
	\end{aligned}
	\right.
\end{equation}
where $x_0\in\mathbb{R}^n$, the coefficients $A(\cdot)$, $B_1(\cdot)$, $B_2(\cdot)$, $C_1(\cdot)$, $C_2(\cdot)$, $D_1(\cdot,\cdot)$ and $D_2(\cdot,\cdot)$ are deterministic, bounded functions with values in $\mathbb{R}^{n\times n}$, $\mathbb{R}^{n\times m}$, $\mathbb{R}^{n\times m}$, $\mathbb{R}^n$, $\mathbb{R}^n$, $\mathbb{R}^n$ and $\mathbb{R}^n$, respectively.

In our problem, we suppose that the state process $X^{u_1,u_2}$ can not be observed directly by the two players. Instead, the follower can observe the process $Y_1^{u_1,u_2}\equiv\{Y_1^{u_1,u_2}(t)\in\mathbb{R};0\leqslant t\leqslant T\}$ and the leader's observed process is $Y_2^{u_1,u_2}\equiv\{Y_2^{u_1,u_2}(t)\in\mathbb{R};0\leqslant t\leqslant T\}$, which evolve according to the following linear SDEPs, respectively:
\begin{equation}\label{observed Y1 follower}
	\left\{
	\begin{aligned}
		dY^{u_1,u_2}_1(t)&=\left[H_1(t)X^{u_1,u_2}(t)+h_{11}(t)u_1(t)+h_{12}(t)u_2(t)\right]dt+K_1(t)dW_1(t)\\
		&\quad +\int_{E_1}f_1(t,e)\tilde{N}_1(de,dt),\\
		Y_1^{u_1,u_2}(0)&=0,
	\end{aligned}
	\right.
\end{equation}
\begin{equation}\label{observed Y2 leader}
	\left\{
	\begin{aligned}
		dY^{u_1,u_2}_2(t)&=\left[H_2(t)X^{u_1,u_2}(t)+h_2(t)u_2(t)\right]dt+K_2(t)dW_2(t)+\int_{E_2}f_2(t,e)\tilde{N}_2(de,dt),\\
		Y_2^{u_1,u_2}(0)&=0,
	\end{aligned}
	\right.
\end{equation}
where the deterministic coefficients $H_1(\cdot), H_2(\cdot)\in\mathbb{R}^{1\times n}$, $K_1(\cdot), K_2(\cdot)\in\mathbb{R}$, $h_{11}(\cdot), h_{12}(\cdot), h_2(\cdot)\in\mathbb{R}^{1\times m}$ and $f_1(\cdot,\cdot), f_2(\cdot,\cdot)\in\mathbb{R}$.

For $i=1,2$, let $\mathcal{Y}_i^{u_1,u_2}\equiv\{\mathcal{Y}_i^{u_1,u_2}(t);0\leqslant t\leqslant T\}$ be the usual augmentation of the filtration generated by $Y_i^{u_1,u_2}$. Clearly, the filtration $\mathcal{Y}_i^{u_1,u_2}$ depends on the choice of $u_1$ and $u_2$. We define the admissible control sets of the follower and the leader, respectively, as follows:
\begin{equation}\label{admissible control sets}
	\begin{aligned}
		\mathcal{U}^F_{ad}&\coloneqq\Bigl\{u_1:[0,T]\times\Omega\to\mathbb{R}^m\Big\vert u_1\mbox{ is }\mathcal{Y}_1^{u_1,u_2}\mbox{-adapted and } \mathbb{E}\int_{0}^{T}\vert u_1(t)\vert dt<\infty\Bigr\},\\
		\mathcal{U}^L_{ad}&\coloneqq\Bigl\{u_2:[0,T]\times\Omega\to\mathbb{R}^m\Big\vert u_2\mbox{ is }\mathcal{Y}_2^{u_1,u_2}\mbox{-adapted and } \mathbb{E}\int_{0}^{T}\vert u_2(t)\vert dt<\infty\Bigr\}.
	\end{aligned}
\end{equation}

In our game within the mean-variance framework, to measure the performance of player's controls $u_i$, we introduce the following cost functionals:
\begin{equation}\label{cost J1 J2 Sec2}
	\begin{aligned}
		J_1(u_1,u_2)&=\mathbb{E}\left[\int_0^T\frac{1}{2}\langle R_1(t)u_1(t),u_1(t)\rangle dt\right]+\mathbb{E}\big[X^{u_1,u_2}_T\big]+\frac{\theta_1}{2}\text{Var}\big[X^{u_1,u_2}_T\big],\\
		J_2(u_1,u_2)&=\mathbb{E}\left[\int_0^T\frac{1}{2}\langle R_2(t)u_2(t),u_2(t)\rangle dt\right]+\mathbb{E}\big[X^{u_1,u_2}_T\big]+\frac{\theta_2}{2}\text{Var}\big[X^{u_1,u_2}_T\big],
	\end{aligned}
\end{equation}
where $\text{Var}$ denoted the variance of a random variable, $\theta_i>0$ $(i=1,2)$ reflects the different degrees of risk aversion of the follower and leader.

\begin{remark}\label{risk averse degree}
	For a random variable $\xi$, consider its exponential utility $F_\theta[\xi]\coloneqq\frac{1}{\theta}\ln\mathbb{E}[e^{\theta\xi}]$, where $\theta$ is a constant. Performing the second-order Taylor expansion to $F_\theta[\xi]$ with respect to $\theta$ around $\theta=0$ gives
	\begin{equation*}
		F_{\theta}[\xi]=\mathbb{E}[\xi]+\frac{\theta}{2}\text{Var}[\xi]+o(\theta).
	\end{equation*}
	 Since $\text{Var}[\xi]$, the volatility, can be regarded as risk, to minimize $F_{\theta}[\xi]$, a decision maker processes the risk-averse (resp. risk-seeking) if $\theta>0$ (resp. $\theta<0$). That's to say, the stochastic Stackelberg differential game with mean-variance cost is a special version of risk-sensitive one. However, if we consider the general risk-sensitive cost functional, i.e.,
$$
J_i(u_1,u_2)=\mathbb{E}\left[\exp\left\{\frac{\theta_i}{2}\int_0^T\langle R_i(t)u_i(t),u_i(t)\rangle dt+\frac{1}{2}\langle\theta_iX^{u_1,u_2}_T,X^{u_1,u_2}_T\rangle\right\}\right],
$$
	 the corresponding risk-sensitive stochastic Stackelberg differential game is rather different and difficult. We do not with to consider this topic in this paper. More recent advances in risk-sensitive stochastic control problems and games can be founded in \cite{LS25,LS25+,MDB19} and references therein.
\end{remark}

Throughout this paper, we impose the following assumptions.
\begin{assumption}\label{assumption state}
	\begin{enumerate}[\bfseries (1)]
		\item $A(\cdot)\in L^\infty(0,T;\mathbb{R}^{n\times n})$, $B_i(\cdot)\in L^\infty(0,T;\mathbb{R}^{n\times m})$, $i=1,2$,\\
		$C_i(\cdot)\in L^\infty(0,T;\mathbb{R}^n)$, $i=1,2$, $D_i(\cdot,\cdot)\in L^\infty(E_i\times[0,T];\mathbb{R}^n)$, $i=1,2$;
		\item $H_i(\cdot)\in L^\infty(0,T;\mathbb{R}^{1\times n})$, $i=1,2$, $h_{11}(\cdot),h_{12}(\cdot),h_2(\cdot)\in L^\infty(0,T;\mathbb{R}^{1\times m})$,\\
        $K_i(\cdot)\in L^\infty(0,T;\mathbb{R})$, $i=1,2$, $f_i(\cdot)\in L^\infty(E_i\times[0,T];\mathbb{R})$, $i=1,2$;
		\item $R_i(\cdot)\in L^\infty(0,T;\mathbb{S}^m)$ and $R_i(\cdot)>0$, $\theta_i>0$, $i=1,2$;
     	\item For $i=1,2$, $K_i(t)$ are invertible and $K_i^{-1}(t)$ are bounded on $[0,T]$;
		\item For $i=1,2$, $f_i(t,e)$ are invertible and $f_i^{-1}(t,e)$ are bounded on $[0,T]\times E_i$.
	\end{enumerate}
\end{assumption}

\section{The follower's problem}\label{Sec follower's}

In this section, we consider the follower's problem, which is a partially observed stochastic linear optimal control problem with special mean-variance criteria. Inspired by \cite{SX23} and \cite{LSZ24}, for a fixed leader's control $u_2\in\mathcal{U}^L_{ad}$, and for any follower's control $u_1\in\mathcal{U}^F_{ad}$, define the following notations
\begin{equation}\label{filter of X wrt Y1}
	\hat{X}^{u_1,u_2}(t)\coloneqq\mathbb{E}\big[X^{u_1,u_2}(t)\big\vert\mathcal{Y}^{u_1,u_2}_1(t)\big],\quad\tilde{X}^{u_1,u_2}(t)\coloneqq X^{u_1,u_2}(t)-\hat{X}^{u_1,u_2}(t),
\end{equation}
where the state process $X^{u_1,u_2}$ satisfies \eqref{state equation 2.1}. As in Lemma \ref{lemma check V 3.1}, we define process $\check{V}^c_1\equiv\{\check{V}^c_1(t);0\leqslant t\leqslant T\}$ as
\begin{equation}
	d\check{V}_1^c(t)\coloneqq K^{-1}_1(t)\big[H_1(t)\tilde{X}^{u_1,u_2}(t)\big]dt+dW_1(t),
\end{equation}
which is a standard $\{\mathcal{Y}^{u_1,u_2}_1(t)\}$-Brownian motion in $\mathbb{R}^d$.

According to Lemma \ref{lemma nonlinear filtering equation}, we have the following result.
\begin{lemma}\label{lemma4.1}
	Let Assumption \ref{assumption state} hold. For a fixed $u_2\in\mathcal{U}^L_{ad}$, and for any $u_1\in\mathcal{U}^F_{ad}$, let $X^{u_1,u_2}$ be the corresponding state process satisfying \eqref{state equation 2.1}. Then $\hat{X}^{u_1,u_2}$ evolves according to
	\begin{equation}\label{state equation follower hat}
		\left\{
		\begin{aligned}
		d\hat{X}^{u_1,u_2}(t)&=\left[A(t)\hat{X}^{u_1,u_2}(t)+B_1(t)u_1(t)+B_2(t)u_2(t)\right]dt\\
		&\quad +\left[\Sigma(t)H_1(t)^\top K^{-1}_1(t)^\top+C_1(t)\right]d\check{V}_1^c(t)+\int_{E_1}D_1(t,e)\tilde{N}_1(de,dt),\\
		\hat{X}^{u_1,u_2}(0)&=x_0,
		\end{aligned}
		\right.
	\end{equation}
and $\tilde{X}^{u_1,u_2}$ satisfies
\begin{equation}\label{state equation follower tilde}
	\left\{
	\begin{aligned}
		d\tilde{X}^{u_1,u_2}(t)&=\left\{A(t)-\left[\Sigma(t)H_1(t)^\top K^{-1}_1(t)^\top+C_1(t)\right]K^{-1}_1(t)H_1(t)\right\}\tilde{X}^{u_1,u_2}(t)dt\\
		&\quad -\Sigma(t)H_1(t)^\top K^{-1}_1(t)^\top dW_1(t)+C_2(t)dW_2(t)+\int_{E_2}D_2(t,e)\tilde{N}_2(de,dt),\\
		\tilde{X}^{u_1,u_2}(0)&=0.
	\end{aligned}
	\right.
\end{equation}
Further, $\Sigma(t)\coloneqq\mathbb{E}\left[\big(\tilde{X}^{u_1,u_2}(t)\big)\big(\tilde{X}^{u_1,u_2}(t)\big)^\top\big\vert\mathcal{Y}^{u_1,u_2}_1(t)\right]$ is given by the following {\it stochastic differential equation} (SDE):
\begin{equation}\label{equation Sigma follower}
	\left\{
	\begin{aligned}
		d\Sigma(t)&=\bigg\{\Sigma(t)\left[A(t)-C_1(t)K^{-1}_1(t)H_1(t)\right]^\top+\left[A(t)-C_1(t)K^{-1}_1(t)H_1(t)\right]\Sigma(t)\\
		&\qquad -\Sigma(t)H_1(t)^\top K^{-1}_1(t)^\top K^{-1}_1(t)H_1(t)\Sigma(t)+C_2(t)C_2(t)^\top\\
        &\qquad +\int_{E_2}D_2(t,e)D_2(t,e)^\top\nu_2(de)\bigg\}dt\\
		&\quad +\mathbb{E}\left[\widetilde{\big(\tilde{X}\tilde{X}^\top\big)}\widetilde{\big(H_1(t)X(t)\big)}^\top\Big\vert\mathcal{Y}^{u_1,u_2}_1(t)\right]K^{-1}_1(t)^\top d\check{V}_1^c(t),\\
		\Sigma(0)&=0.
	\end{aligned}
	\right.
\end{equation}
\end{lemma}

\begin{remark}
	In our setting, $\{\Sigma(t);0\leqslant t\leqslant T\}$, the conditional covariance matrix of estimate error $\tilde{X}^{u_1,u_2}$, is random rather than deterministic. In classical continuous linear filtering, which also be known as Kalman-Bucy filtering, this covariance matrix satisfies an {\it ordinary differential equation} (ODE). Actually, if there is no Poisson random measure in the system, it is shown that the state process $X^{u_1,u_2}$ and the observation process $Y^{u_1,u_2}$ are jointly Guassian, which leads to the property that $\mathbb{E}\big[\big(\tilde{X}^{u_1,u_2}(t)\big)\big(\tilde{X}^{u_1,u_2}(t)\big)^\top\big\vert\mathcal{Y}^{u_1,u_2}_1(t)\big]=\mathbb{E}\big[\big(\tilde{X}^{u_1,u_2}(t)\big)\big(\tilde{X}^{u_1,u_2}(t)\big)^\top\big]$. See Chapter 7.5 in \cite{Situ05} for more details.
\end{remark}

\begin{proof}	
Noting that $u_1$ and $u_2$ are both $\mathcal{Y}^{u_1,u_2}_1$-adapted, the linear filter equation \eqref{state equation follower hat} follows immediately from Lemma \ref{lemma nonlinear filtering equation}. It is not hard to get \eqref{state equation follower tilde} since $\tilde{X}^{u_1,u_2}=X^{u_1,u_2}-\hat{X}^{u_1,u_2}$.

By It\^o's formula, we have
\begin{equation*}
	\begin{aligned}
        &d\big[\big(\tilde{X}^{u_1,u_2}(t)\big)\big(\tilde{X}^{u_1,u_2}(t)\big)^\top\big]\\
        &=\big(\tilde{X}^{u_1,u_2}(t)\big)\big(\tilde{X}^{u_1,u_2}(t)\big)^\top\left\{A(t)-\left[\Sigma(t)H_1(t)^\top K^{-1}_1(t)^\top+C_1(t)\right]K^{-1}_1(t)H_1(t)\right\}^\top dt\\
		&\quad -\tilde{X}^{u_1,u_2}(t)\big(dW_1(t)\big)^\top\left[\Sigma(t)H_1(t)^\top K^{-1}_1(t)^\top\right]^\top+\tilde{X}^{u_1,u_2}(t)\big(dW_2(t)\big)^\top C_2(t)^\top\\
		&\quad +\int_{E_2}\tilde{X}^{u_1,u_2}(t)D_1(t,e)^\top\tilde{N}_2(de,dt)\\
		&\quad +\left\{A(t)-\left[\Sigma(t)H_1(t)^\top K^{-1}_1(t)^\top+C_1(t)\right]K^{-1}_1(t)H_1(t)\right\}\big(\tilde{X}^{u_1,u_2}(t)\big)\big(\tilde{X}^{u_1,u_2}(t)\big)^\top dt\\
		&\quad -\left[\Sigma(t)H_1(t)^\top K^{-1}_1(t)^\top\right]dW_1(t)\big(\tilde{X}^{u_1,u_2}(t)\big)^\top+C_2(t)dW_2(t)\big(\tilde{X}^{u_1,u_2}(t)\big)^\top\\
		&\quad +\int_{E_2}D_2(t,e)\big(\tilde{X}^{u_1,u_2}(t)\big)^\top\tilde{N}_2(de,dt)\\
		&\quad +\left[\Sigma(t)H_1(t)^\top K^{-1}_1(t)^\top\right]dW_1(t)\big(dW_1(t)\big)^\top\left[\Sigma(t)H_1(t)^\top K^{-1}_1(t)^\top\right]^\top\\
		&\quad +C_2(t)dW_2(t)\big(dW_2(t)\big)^\top C_2(t)^\top+\int_{E_2}D_2(t,e)D_2(t,e)^\top N_2(de,dt).
	\end{aligned}
\end{equation*}
Again, using Lemma \ref{lemma nonlinear filtering equation}, we arrive at the non-linear filter equation \eqref{equation Sigma follower}.
\end{proof}

It is worthy to mention that $\tilde{X}^{u_1,u_2}$ is actually independent of controls $u_1$ and $u_2$. Further, for each $t\in[0,T]$, $\hat{X}^{u_1,u_2}(t)$ is the orthogonal projection of $X^{u_1,u_2}(t)$ onto $L^2(\Omega,\mathcal{Y}^{u_1,u_2}_1(t))$, the Hilbert space of $\mathcal{Y}^{u_1,u_2}_1(t)$-measurable, square-integrable random variables. In this case, $\tilde{X}^{u_1,u_2}(t)$ is independent of $\mathcal{Y}^{u_1,u_2}_1(t)$ and
\begin{equation}\label{orthogonal decomposition}
\begin{aligned}
	\mathbb{E}\big\vert X^{u_1,u_2}(t)\big\vert^2&=\mathbb{E}\big\vert \hat{X}^{u_1,u_2}(t)\big\vert^2+\mathbb{E}\big\vert \tilde{X}^{u_1,u_2}(t)\big\vert^2,\\ \text{Var}[X^{u_1,u_2}(t)]&=\text{Var}[\hat{X}^{u_1,u_2}(t)]+\text{Var}[\tilde{X}^{u_1,u_2}(t)],\quad \text{for all } t\in[0,T].
\end{aligned}
\end{equation}
Thus, the cost functional $J_1(u_1,u_2)$ in \eqref{cost J1 J2 Sec2} can be decomposed as follows
\begin{equation}\label{decom J_1}
	J_1(u_1,u_2)=\hat{J}_1(u_1,u_2)+\tilde{J}_1(u_1,u_2),
\end{equation}
where
\begin{equation}\label{hat J1 and tilde J1}
	\left\{
	\begin{aligned}
		\hat{J}_1(u_1,u_2)&\coloneqq\mathbb{E}\left[\int_0^T\frac{1}{2}\langle R_1(t)u_1(t),u_1(t)\rangle dt\right]+\mathbb{E}\big[\hat{X}^{u_1,u_2}(T)\big]+\frac{\theta_1}{2}\text{Var}\big[\hat{X}^{u_1,u_2}(T)\big],\\
		\tilde{J}_1(u_1,u_2)&\coloneqq\mathbb{E}\big[\tilde{X}^{u_1,u_2}(T)\big]+\frac{\theta_1}{2}\text{Var}\big[\tilde{X}^{u_1,u_2}(T)\big].
	\end{aligned}
	\right.
\end{equation}

Now, we have decomposed the follower problem, a partially observed stochastic optimal control problem, into two fully observed problems. Especially, $\hat{J}_1(u_1,u_2)$ is a mean-variance criteria and $\tilde{J}_1(u_1,u_2)$ is independent of controls $u_1$ and $u_2$.

Next, following the method of handling the mean-variance problem in \cite{ZL00}, we now transform the problem with respect to $\hat{J}_1(u_1,u_2)$ into an SLQ optimal control problem.
\begin{equation*}
\hspace{-2.5cm} \mathbf{P_F(\theta_1):}\ \mbox{For any given $u_2$, minimize $\hat{J}_1(u_1,u_2)$ over $u_1\in\mathcal{U}^F_{ad}$, subject to \eqref{state equation follower hat}}.
\end{equation*}

Define the optimal control set of problem $P_F(\theta_1)$ as:
\begin{equation*}
	\Pi_{P_F(\theta_1)}\coloneqq\bigl\{u_1\big\vert u_1 \mbox{ is an optimal control of }P_F(\theta_1)\bigr\}.
\end{equation*}

We now propose to embed problem $P_F(\theta_1)$ into a tractable auxiliary problem that turns out to be an SLQ optimal control problem. To do this, we need the following cost functional:
\begin{equation}\label{auxiliary cost functional}
	\hat{\mathcal{J}}_1(u_1,u_2)\coloneqq\mathbb{E}\left[\int_0^T\frac{1}{2}\langle R_1(t)u_1(t),u_1(t)\rangle dt+\lambda_1\hat{X}^{u_1,u_2}(T)+\frac{\theta_1}{2}|\hat{X}^{u_1,u_2}(T)|^2\right].
\end{equation}
Then the auxiliary problem $A_F(\theta_1,\lambda_1)$ is given by
\begin{equation*}
\hspace{-5mm} \mathbf{A_F(\theta_1,\lambda_1):}\ \mbox{For any given } u_2, \mbox{minimize }\hat{\mathcal{J}}_1(u_1,u_2)\mbox{ over }u_1\in\mathcal{U}^F_{ad}, \mbox{ subject to \eqref{state equation follower hat}}.
\end{equation*}
And define the optimal control set of problem $A_F(\theta_1,\lambda_1)$ as:
\begin{equation*}
	\Pi_{A_F(\theta_1,\lambda_1)}\coloneqq\bigl\{u_1\big\vert u_1 \mbox{ is an optimal control of }A_F(\theta_1,\lambda_1)\bigr\}.
\end{equation*}

The following result shows the relationship between problems $P_F(\theta_1)$ and $A_F(\theta_1,\lambda_1)$.
\begin{lemma}\label{embedding theorem}
{\bf (Embedding theorem of the follower)}
For any $\lambda_1>0$, one has
\begin{equation}\label{embedding result}
		\Pi_{P_F(\theta_1)}\subset\bigcup_{-\infty<\lambda_1<+\infty}\Pi_{A_F(\theta_1,\lambda_1)}.
\end{equation}
Moreover, if $\bar{u}_1\in\Pi_{P_F(\theta_1)}$, then $\bar{u}_1\in\Pi_{A_F(\theta_1,\bar{\lambda}_1)}$ with $\bar{\lambda}_1=1-\theta_1\mathbb{E}[\hat{X}^{\bar{u}_1,u_2}(T)]$, where $\hat{X}^{\bar{u}_1,u_2}$ is the corresponding trajectory.
\end{lemma}

\begin{proof}
	We only need to prove the second part as the first one is a direct consequence of the second. Let $\bar{u}_1\in\Pi_{P_F(\theta_1)}$. If $\bar{u}_1\notin\Pi_{A_F(\theta_1,\bar{\lambda}_1)}$, then there exist $u_1$ and the corresponding $X^{u_1,u_2}$ such that
\begin{equation*}
	\begin{aligned}
		&\mathbb{E}\left[\frac{1}{2}\int_0^T\left[\langle R_1(t)u_1(t),u_1(t)\rangle-\langle R_1(t)\bar{u}_1(t),\bar{u}_1(t)\rangle\right] dt\right]
        +\frac{\theta_1}{2}\Big[\mathbb{E}|X^{u_1,u_2}(T)|^2-\mathbb{E}|X^{\bar{u}_1,u_2}(T)|^2\Big]\\
		&-\big(\theta_1\mathbb{E}\big[X^{\bar{u}_1,u_2}(T)\big]-1\big)\Big[\mathbb{E}\big[X^{u_1,u_2}(T)\big]-\mathbb{E}\big[X^{\bar{u}_1,u_2}(T)\big]\Big]\leqslant0.
	\end{aligned}		
\end{equation*}
Set a function $\pi(x,y,v)\coloneqq\frac{\theta_1}{2}x-\frac{\theta_1}{2}y^2+y+v$ which is concave in $(x,y,v)$. The cost functional of problem $P_F(\theta_1)$ can be represented as
\begin{equation*}
	\hat{J}_1(u_1,u_2)=\pi\left(\mathbb{E}|X^{u_1,u_2}(T)|^2,\mathbb{E}\big[X^{u_1,u_2}(T)\big],\mathbb{E}\left[\int_0^T\frac{1}{2}\langle R_1(t)u_1(t),u_1(t)\rangle dt\right]\right).
\end{equation*}
Noting $\frac{\partial \pi}{\partial x}(x,y,v)=\frac{\theta_1}{2}$, $\frac{\partial \pi}{\partial y}(x,y,v)=(1-\theta_1y)$ and $\frac{\partial \pi}{\partial v}(x,y,v)=1$, the concavity of $\pi$ gives
\begin{equation*}
	\begin{aligned}
		&\pi\left(\mathbb{E}|X^{u_1,u_2}(T)|^2,\mathbb{E}\big[X^{u_1,u_2}(T)\big],\mathbb{E}\left[\int_0^T\frac{1}{2}\langle R_1(t)u_1(t),u_1(t)\rangle dt\right]\right)\\
		&\leqslant \pi\left(\mathbb{E}|X^{\bar{u}_1,u_2}(T)|^2,\mathbb{E}\big[X^{\bar{u}_1,u_2}(T)\big],\mathbb{E}\left[\int_0^T\frac{1}{2}\langle R_1(t)\bar{u}_1(t),u_1(t)\rangle dt\right]\right)\\
		&\quad +\mathbb{E}\left[\frac{1}{2}\int_0^T\left[\langle R_1(t)u_1(t),u_1(t)\rangle-\langle R_1(t)\bar{u}_1(t),\bar{u}_1(t)\rangle\right] dt\right]\\
        &\quad +\frac{\theta_1}{2}\Big[\mathbb{E}|X^{u_1,u_2}(T)|^2-\mathbb{E}|X^{\bar{u}_1,u_2}(T)|^2\Big]\\
        &\quad -\big(\theta_1\mathbb{E}\big[X^{\bar{u}_1,u_2}(T)\big]-1\big)\Big[\mathbb{E}\big[X^{u_1,u_2}(T)\big]-\mathbb{E}\big[X^{\bar{u}_1,u_2}(T)\big]\Big]\\
		&\leqslant \pi\left(\mathbb{E}|X^{\bar{u}_1,u_2}(T)|^2,\mathbb{E}\big[X^{\bar{u}_1,u_2}(T)\big],\mathbb{E}\left[\int_0^T\frac{1}{2}\langle R_1(t)\bar{u}_1(t),\bar{u}_1(t)\rangle dt\right]\right),
	\end{aligned}
\end{equation*}
which is contradict to $\bar{u}_1\in\Pi_{P_F(\theta_1)}$.
\end{proof}

To deal with $\hat{\mathcal{J}}_1(u_1,u_2)$, we introduce a Riccati equation of $P(\cdot)\in\mathbb{R}^{n\times n}$:
\begin{equation}\label{equation P follower}
	\left\{
	\begin{aligned}
		&\dot{P}(t)+P(t)A(t)+A(t)^\top P(t)-P(t)^\top B_1(t)R^{-1}_1(t)B_1(t)^\top P(t)=0,\\
		&P(T)=\theta_1,
	\end{aligned}
	\right.
\end{equation}
and a {\it backward SDEP} (BSDEP) of $(\varphi^{u_2},\phi^{u_2},\psi^{u_2})\equiv\{\varphi^{u_2}(t)\in\mathbb{R}^n,\psi^{u_2}(t)\in\mathbb{R}^n,\psi^{u_2}(t,e)\in\mathbb{R}^n;0\leqslant t\leqslant T,e\in E_1\}$:
\begin{equation}\label{equation varphi follower}
	\left\{
	\begin{aligned}
		-d\varphi^{u_2}(t)&=\left\{\left[A(t)^\top-P(t)B_1(t)R_1^{-1}(t)B_1(t)^\top\right]\varphi^{u_2}(t)+P(t)B_2(t)u_2(t)\right\}dt\\
		&\quad -\phi^{u_2}(t)d\check{V}^c(t)-\int_{E_1}\psi^{u_2}(t,e)\tilde{N}_1(de,dt),\\
		\varphi^{u_2}(T)&=\lambda_1.
	\end{aligned}
	\right.
\end{equation}

The following result is standard.
\begin{lemma}
	Let Assumption \ref{assumption state} hold. For a fixed $u_2\in\mathcal{U}^L_{ad}$, suppose \eqref{equation P follower} and \eqref{equation varphi follower} admit adapted solutions $P(\cdot)$ and $(\varphi^{u_2},\phi^{u_2},\psi^{u_2})$, respectively. Then the follower's problem is solvable with the optimal control $\bar{u}_1$ given by
	\begin{equation}\label{optimal control of follower u1}
		\bar{u}_1(t)=-R_1^{-1}(t)B_1(t)^\top\left[P(t)\hat{X}^{\bar{u}_1,u_2}(t)+\varphi^{u_2}(t)\right],\quad t\in[0,T],
	\end{equation}
where $\hat{X}^{\bar{u}_1,u_2}$ satisfies
	\begin{equation}\label{filter equation of X wrt Y1}
		\left\{
		\begin{aligned}
		d\hat{X}^{\bar{u}_1,u_2}(t)&=\left[A(t)\hat{X}^{\bar{u}_1,u_2}(t)-B_1(t)R_1^{-1}(t)B_1(t)^\top\left[P(t)\hat{X}^{\bar{u}_1,u_2}(t)+\varphi^{u_2}(t)\right]+B_2(t)u_2(t)\right]dt\\
		&\quad +\left[\Sigma(t)H_1(t)^\top K^{-1}_1(t)^\top+C_1(t)\right]d\check{V}_1^c(t)+\int_{E_1}D_1(t,e)\tilde{N}_1(de,dt),\\
		\hat{X}^{\bar{u}_1,u_2}(0)&=x_0,
		\end{aligned}
		\right.
	\end{equation}
\end{lemma}

\begin{proof}
Since $\tilde{J}_1(u_1,u_2)$ defined in \eqref{decom J_1} is independent of $u_1$, the optimal control $\bar{u}_1$ of the follower's problem is also optimal to $P_F(\theta_1)$, as well as the auxiliary problem $A_F(\theta_1,\bar{\lambda}_1)$ with $\bar{\lambda}_1=1-\theta_1\mathbb{E}[\hat{X}^{\bar{u}_1,u_2}(T)]$.
	
By applying It\^o's formula to $\frac{1}{2}\langle P(\cdot)\hat{X}^{u_1,u_2}(\cdot),\hat{X}^{u_1,u_2}(\cdot)\rangle$ and $\langle\varphi(\cdot),\hat{X}^{u_1,u_2}(\cdot)\rangle$ and inserting them into $\hat{\mathcal{J}}_1(u_1,u_2)$, we obtain that
\begin{equation*}
	\begin{aligned}
		&\hat{\mathcal{J}}_1(u_1,u_2)=\frac{1}{2}\big\langle P(0)\hat{X}^{u_1,u_2}(0),\hat{X}^{u_1,u_2}(0)\big\rangle +\big\langle\varphi^{u_2}(0),\hat{X}^{u_1,u_2}(0)\big\rangle\\
		&\quad +\mathbb{E}\bigg[\int_0^T\bigg\{\frac{1}{2}\left\langle R_1\left[u_1+R_1^{-1}B_1^\top\big(P\hat{X}^{u_1,u_2}+\varphi^{u_2}\big)\right],
          \left[u_1+R_1^{-1}B_1^\top\big(P\hat{X}^{u_1,u_2}+\varphi^{u_2}\big)\right]\right\rangle\\
		&\quad -\frac{1}{2}\big\langle R_1^{-1}B_1^\top\varphi^{u_2},B_1^\top\varphi^{u_2}\big\rangle +\frac{1}{2}\left\langle P\big(\Sigma H_1^\top (K_1^{-1})^\top+C_1\big),\Sigma H_1^\top (K_1^{-1})^\top+C_1\right\rangle \\
		&\quad +\big\langle\varphi^{u_2},B_2u_2\big\rangle +\big\langle\phi^{u_2},\Sigma H_1^\top (K_1^{-1})^\top+C_1\big\rangle +\int_{E_1}\big\langle PD_1+\psi^{u_2},D_1\big\rangle\nu_1(de)\bigg\}dt\bigg]\\
        &\geqslant \frac{1}{2}\big\langle P(0)\hat{X}^{u_1,u_2}(0),\hat{X}^{u_1,u_2}(0)\big\rangle +\big\langle\varphi^{u_2}(0),\hat{X}^{u_1,u_2}(0)\big\rangle\\
		&\quad +\mathbb{E}\bigg[\int_0^T\bigg\{-\frac{1}{2}\big\langle R_1^{-1}B_1^\top\varphi^{u_2},B_1^\top\varphi^{u_2}\big\rangle
         +\frac{1}{2}\left\langle P\big(\Sigma H_1^\top (K_1^{-1})^\top+C_1\big),\Sigma H_1^\top (K_1^{-1})^\top+C_1\right\rangle\\
		&\quad +\big\langle\varphi^{u_2},B_2u_2\big\rangle +\big\langle\phi^{u_2},\Sigma H_1^\top (K_1^{-1})^\top+C_1\big\rangle +\int_{E_1}\big\langle PD_1+\psi^{u_2},D_1\big\rangle\nu_1(de)\bigg\}dt\bigg]\\
		&=\hat{\mathcal{J}}_1(\bar{u}_1,u_2),
	\end{aligned}
\end{equation*}
which implies that $\bar{u}_1$ is optimal.
\end{proof}

\section{The leader's problem}\label{Sec leader's}

In this section, we are going to study the stochastic optimal control problem of the leader. Knowing that the follower would take her/his optimal control $\bar{u}_1$ in \eqref{optimal control of follower u1}, the leader's state process $X^{u_2}\coloneqq X^{\bar{u}_1,u_2}$ will depends on $\hat{X}^{u_2}\coloneqq\hat{X}^{\bar{u}_1,u_2}$ and $\varphi^{u_2}\coloneqq\varphi^{\bar{u}_1,u_2}$. Thus, the state equation of the leader becomes an {\it FBSDE with Poisson random measure} (FBSDEP).

Especially, in this section, we need to set $H_1(\cdot)\equiv0$ to obtain some explicit results. The general case remains open. Now, the state equation of the leader's problem becomes
\begin{equation}\label{state equation leader}
	\left\{
	\begin{aligned}
		dX^{u_2}(t)=&\left[A(t)X^{u_2}(t)-B_1(t)R_1(t)^{-1}B_1(t)^\top\big(P(t)\hat{X}^{u_2}(t)+\varphi^{u_2}(t)\big)+B_2(t)u_2(t)\right]dt\\
		&+\sum_{i=1}^2C_i(t)dW_i(t)+\sum_{i=1}^2\int_{E_i}D_i(t,e)\tilde{N}_i(de,dt),\\
		-d\varphi^{u_2}(t)=&\left\{\left[A(t)^\top-P(t)B_1(t)R_1^{-1}(t)B_1(t)^\top\right]\varphi^{u_2}(t)+P(t)B_2(t)u_2(t)\right\}dt\\
		&-\phi^{u_2}(t)dW_1(t)-\int_{E_1}\psi^{u_2}(t,e)\tilde{N}_1(de,dt),\\
		X^{u_2}(0)=&\ x_0,\quad\varphi^{u_2}(T)=\bar{\lambda}_1,
	\end{aligned}
	\right.
\end{equation}
where $\hat{X}^{u_2}$ satisfied \eqref{filter equation of X wrt Y1}.

\subsection{Filtering equations}

In this subsection, we first derive some filtering equations by some lemmas in the appendix.

Let $\mathcal{Y}^{u_2}_2\coloneqq \mathcal{Y}^{\bar{u}_1,u_2}_2$ be the filtration generated by the leader's observation $Y^{u_2}_2\coloneqq Y^{\bar{u}_1,u_2}_2$. We also assume that $\mathcal{Y}^{u_2}_2(t)\subset\mathcal{Y}^{\bar{u}_1,u_2}_1(t)$ for any $t\in[0,T]$ in this paper. For any $u_2\in\mathcal{U}^L_{ad}$, set
\begin{equation*}
	\begin{aligned}
		&\check{X}^{u_2}(t)\coloneqq\mathbb{E}\big[X^{u_2}(t)\big\vert\mathcal{Y}^{u_2}_2(t)\big],\quad
        \tilde{\tilde{X}}^{u_2}(t)\coloneqq X^{u_2}(t)-\check{X}^{u_2}(t),\quad
        \check{\varphi}^{u_2}(t)\coloneqq\mathbb{E}\big[\varphi^{u_2}(t)\big\vert\mathcal{Y}^{u_2}_2(t)\big],\\	
        &\tilde{\tilde{\varphi}}^{u_2}(t)\coloneqq\varphi^{u_2}(t)-\check{\varphi}^{u_2}(t),\quad
        \check{\phi}^{u_2}(t)\coloneqq\mathbb{E}\big[\phi^{u_2}(t)\big\vert\mathcal{Y}^{u_2}_2(t)\big],\quad
        \check{\psi}^{u_2}(t,e)\coloneqq\mathbb{E}\big[\psi^{u_2}(t,e)\big\vert\mathcal{Y}^{u_2}_2(t)\big].
	\end{aligned}
\end{equation*}

Introduce the innovation process
\begin{equation}
	U(t)\coloneqq Y^{u_2}_2(t)-\int_0^t\big[H_2(s)\check{X}^{u_2}(s)+h_2(s)u_2(s)\big]ds
\end{equation}
and the standard ${\mathcal{Y}^{u_2}_2(t)}$-Brownian motion
\begin{equation}
	\check{U}^c(t)\coloneqq\int_0^tK^{-1}_2(s)H_2(s)\tilde{\tilde{X}}^{u_2}(s)ds+\int_0^tdW_2(s).
\end{equation}

The following result is an application of Lemma \ref{lemma nonlinear filtering equation}.

\begin{lemma}
Let Assumption \ref{assumption state} hold. For any $u_2\in\mathcal{U}^L_{ad}$, $(\check{X}^{u_2},\check{\varphi}^{u_2})$, the filtering of state process $(X^{u_2},\varphi^{u_2})$ corresponding to the leader's observation process $Y^{u_2}_2$, evolves according to	
\begin{equation}\label{state equation leader check X}
	\left\{
	\begin{aligned}
		d\check{X}^{u_2}(t)=&\Big\{\big[A(t)-B_1(t)R_1^{-1}(t)B_1(t)^\top P(t)\big]\check{X}^{u_2}(t)-B_1(t)R_1^{-1}(t)B_1(t)^\top\check{\varphi}^{u_2}(t)\\
        &\quad +B_2(t)u_2(t)\Big\}dt +\Big\{\Xi_1(t)H_2(t)^\top K^{-1}_2(t)^\top+C_2(t)\Big\}d\check{U}^c(t)\\
		&+\int_{E_2}D_2(t,e)\tilde{N}_2(de,dt),\\
		\check{X}^{u_2}(0)=&\ x_0,
	\end{aligned}
	\right.
\end{equation}
where $\Xi_1(t)\coloneqq\mathbb{E}\Big[\tilde{\tilde{X}}^{u_2}(t)\Big(\tilde{\tilde{X}}^{u_2}(t)\Big)^\top\Big\vert\mathcal{Y}^{u_2}_2(t)\Big]$, and
\begin{equation}\label{state equation leader check varphi}
	\left\{
	\begin{aligned}
		-d\check{\varphi}^{u_2}(t)=&\left\{\big[A(t)^\top-P(t)B_1(t)R_1^{-1}(t)B_1(t)^\top\big]\check{\varphi}^{u_2}(t)+P(t)B_2(t)u_2(t)\right\}dt\\
		&+\Xi_2(t)H_2(t)^\top K^{-1}_2(t)^\top d\check{U}^c(t),\\
		\check{\varphi}^{u_2}(T)=&\ \bar{\lambda}_1,
	\end{aligned}
	\right.
\end{equation}
where $\Xi_2(t)\coloneqq\mathbb{E}\Big[\tilde{\tilde{\varphi}}^{u_2}(t)\Big(\tilde{\tilde{X}}^{u_2}(t)\Big)^\top\Big\vert\mathcal{Y}^{u_2}_2(t)\Big]$.
\end{lemma}

\begin{remark}
As in Section \ref{Sec follower's}, we will transform the leader's problem into a fully observed SLQ optimal control problem, and \eqref{state equation leader check varphi} is one of the state equation. If we do not set $H_1(\cdot)=0$, then, noting $d\check{V}^c=K^{-1}_1H_1\tilde{X}^{u_2}dt+dW_1$, a special term $\mathbb{E}\big[\phi^{u_2}(t)K^{-1}_1(t)H_1(t)\tilde{X}^{u_2}(t)\big\vert \mathcal{Y}^{u_2}_2(t)\big]$ will appear in \eqref{state equation leader check varphi}, which will pose great difficulties for leader's problem. As a result, when it comes to leader's problem, we need the assumption that $H_1(\cdot)=0$. The general case is challenging.
\end{remark}

\begin{lemma}
Let Assumption \ref{assumption state} hold. For any $u_2\in\mathcal{U}^L_{ad}$, the difference $\big(\tilde{\tilde{X}}^{u_2}(\cdot),\tilde{\tilde{\varphi}}^{u_2}(\cdot)\big)$ evolves according to
\begin{equation}\label{equation tilde tilde X}
	\left\{
	\begin{aligned}
		d\tilde{\tilde{X}}^{u_2}(t)=&\bigg\{\Big[A(t)-\Big[\Xi_1(t)H_2(t)^\top K^{-1}_2(t)^\top+C_2(t)\Big]K^{-1}_2(t)H_2(t)\Big]\tilde{\tilde{X}}^{u_2}(t)\\
		&\ -B_1(t)R_1^{-1}(t)B_1(t)^\top P(t)\big(\hat{X}^{u_2}(t)-\check{X}^{u_2}_t\big)-B_1(t)R_1^{-1}(t)B_1(t)^\top\tilde{\tilde{\varphi}}^{u_2}(t)\bigg\}dt\\
		&+C_1(t)dW_1(t)-\Big[\Xi_1(t)H_2(t)^\top K^{-1}_2(t)^\top+C_2(t)\Big]dW_2(t)+\int_{E_1}D_1(t,e)\tilde{N}_1(de,dt),\\
		\tilde{\tilde{X}}_0=&\ 0,
	\end{aligned}
	\right.
\end{equation}
and
\begin{equation}\label{equation tilde tilde variphi}
	\left\{
	\begin{aligned}
		-d\tilde{\tilde{\varphi}}^{u_2}(t)=&\bigg\{\Big[A(t)^\top-P(t)B_1(t)R_1^{-1}(t)B_1(t)^\top\Big]\tilde{\tilde{\varphi}}^{u_2}(t)\\
		&\ +\Xi_2(t)H_2(t)^\top K^{-1}_2(t)^\top K^{-1}_2(t)H_2(t)\tilde{\tilde{X}}^{u_2}_t\bigg\}dt\\
		&-\phi^{u_2}(t)dW_1(t)+\Xi_2(t)H_2(t)^\top K^{-1}_2(t)^\top dW_2(t)-\int_{E_1}\psi^{u_2}(t,e)\tilde{N}_1(de,dt),\\
		\tilde{\tilde{\varphi}}^{u_2}(T)=&\ 0.
	\end{aligned}
	\right.
\end{equation}
\end{lemma}

Define
\begin{equation*}
	\begin{aligned}		
\Xi_3(t)&\coloneqq\mathbb{E}\Big[\tilde{\tilde{\varphi}}^{u_2}(t)\Big(\tilde{\tilde{\varphi}}^{u_2}(t)\Big)^\top\Big\vert\mathcal{Y}^{u_2}_2(t)\Big],\quad
\Xi_4(t)\coloneqq\mathbb{E}\Big[\big(\hat{X}^{u_2}(t)-\check{X}^{u_2}(t)\big)\Big(\tilde{\tilde{X}}^{u_2}(t)\Big)^\top\Big\vert\mathcal{Y}^{u_2}_2(t)\Big],\\		\Xi_5(t)&\coloneqq\mathbb{E}\Big[\big(\hat{X}^{u_2}(t)-\check{X}^{u_2}(t)\big)\Big(\tilde{\tilde{\varphi}}^{u_2}(t)\Big)^\top\Big\vert\mathcal{Y}^{u_2}_2(t)\Big],\\
\Xi_6(t)&\coloneqq\mathbb{E}\Big[\big(\hat{X}^{u_2}(t)-\check{X}^{u_2}(t)\big)\big(\hat{X}^{u_2}(t)-\check{X}^{u_2}(t)\big)^\top\Big\vert\mathcal{Y}^{u_2}_2(t)\Big],\\
	\end{aligned}
\end{equation*}
and use the following notations (For simplicity, we omit the time variable $t$.)
\begin{equation*}
	\begin{aligned}		
     &\widetilde{\Big[\tilde{\tilde{X}}^{u_2}\big(\tilde{\tilde{X}}^{u_2}\big)^\top\Big]}\coloneqq\tilde{\tilde{X}}^{u_2}\big(\tilde{\tilde{X}}^{u_2}\big)^\top
     -\mathbb{E}\Big[\tilde{\tilde{X}}^{u_2}\big(\tilde{\tilde{X}}^{u_2}\big)^\top\Big\vert\mathcal{Y}^{u_2}_2\Big],\\
     &\widetilde{\Big[\tilde{\tilde{\varphi}}^{u_2}\big(\tilde{\tilde{X}}^{u_2}\big)^\top\Big]}\coloneqq\tilde{\tilde{\varphi}}^{u_2}\big(\tilde{\tilde{X}}^{u_2}\big)^\top
     -\mathbb{E}\Big[\tilde{\tilde{\varphi}}^{u_2}\big(\tilde{\tilde{X}}^{u_2}\big)^\top\Big\vert\mathcal{Y}^{u_2}_2\Big],\\\
	 &\widetilde{\Big[\tilde{\tilde{\varphi}}^{u_2}\big(\tilde{\tilde{\varphi}}^{u_2}\big)^\top\Big]}\coloneqq\tilde{\tilde{\varphi}}^{u_2}\big(\tilde{\tilde{\varphi}}^{u_2}\big)^\top
     -\mathbb{E}\Big[\tilde{\tilde{\varphi}}^{u_2}\big(\tilde{\tilde{\varphi}}^{u_2}\big)^\top\Big\vert\mathcal{Y}^{u_2}_2\Big],\\
     &\widetilde{\Big[\big(\hat{X}^{u_2}-\check{X}^{u_2}\big)\big(\tilde{\tilde{X}}^{u_2}\big)^\top\Big]}\coloneqq\big(\hat{X}^{u_2}-\check{X}^{u_2}\big)\big(\tilde{\tilde{X}}^{u_2}\big)^\top
     -\mathbb{E}\Big[\big(\hat{X}^{u_2}-\check{X}^{u_2}\big)\big(\tilde{\tilde{X}}^{u_2}\big)^\top\Big\vert\mathcal{Y}^{u_2}_2\Big],\\
     &\widetilde{\Big[\big(\hat{X}^{u_2}-\check{X}^{u_2}\big)\big(\tilde{\tilde{\varphi}}^{u_2}\big)^\top\Big]}\coloneqq\big(\hat{X}^{u_2}-\check{X}^{u_2}\big)\big(\tilde{\tilde{\varphi}}^{u_2}\big)^\top
     -\mathbb{E}\Big[\big(\hat{X}^{u_2}-\check{X}^{u_2}\big)\big(\tilde{\tilde{\varphi}}^{u_2}\big)^\top\Big\vert\mathcal{Y}^{u_2}_2\Big],\\
     &\widetilde{\Big[\big(\hat{X}^{u_2}-\check{X}^{u_2}\big)\big(\hat{X}^{u_2}-\check{X}^{u_2}\big)^\top\Big]}\coloneqq\big(\hat{X}^{u_2}-\check{X}^{u_2}\big)\big(\hat{X}^{u_2}-\check{X}^{u_2}\big)^\top\\
     &\hspace{5.6cm} -\mathbb{E}\Big[\big(\hat{X}^{u_2}-\check{X}^{u_2}\big)\big(\hat{X}^{u_2}-\check{X}^{u_2}\big)^\top\Big\vert\mathcal{Y}^{u_2}_2\Big].
	\end{aligned}
\end{equation*}

The following coupled representations of $\Xi_1(\cdot)$ and $\Xi_2(\cdot)$ are direct results of It\^o's formula and Lemma \ref{lemma nonlinear filtering equation}.
\begin{lemma}
Let Assumption \ref{assumption state} hold. For any $u_2\in\mathcal{U}^L_{ad}$, $\Xi_1(\cdot)$ and $\Xi_2(\cdot)$ satisfy	
	\begin{equation}\label{Xi 1}
		\begin{aligned}
			\Xi_1(t)=&\int_0^t\Big\{\Xi_1A^\top+A\Xi_1-\Xi_4^\top PB_1R_1^{-1}B_1^\top -B_1R_1^{-1}B_1^\top (\Xi_4-\Xi_2^\top B_1R_1^{-1}B_1^\top\\
			&\qquad -B_1R_1^{-1}B_1^\top\Xi_2-\Xi_1H_2^\top(K^{-1}_2)^\top K^{-1}_2H_2\Xi_1 +C_2C_2^\top\Big\}ds\\			
            &+\int_0^t\mathbb{E}\Big[\widetilde{\Big[\tilde{\tilde{X}}^{u_2}\big(\tilde{\tilde{X}}^{u_2}\big)^\top\Big]}\big(\tilde{\tilde{X}}^{u_2}\big)^\top\Big\vert\mathcal{Y}_2^{u_2}\Big]H_2^\top(K^{-1}_2)^\top d\check{U}^c(s),
		\end{aligned}
	\end{equation}	
and
	\begin{equation}\label{Xi 2}
		\begin{aligned}
			\Xi_2(t)=&\int_0^t\Big\{\Xi_2A^\top-A^\top\Xi_2-\Xi_5^\top PB_1R_1^{-1}B_1^\top -\Xi_3B_1R_1^{-1}B_1^\top+PB_1R_1^{-1}B_1^\top\Xi_2\\
			&\qquad -\Xi_2H_2^\top (K^{-1}_2)^\top K^{-1}_2H_2\Xi_1+\check{\phi}^{u_2}C_1^\top+\int_{E_1}\check{\psi}^{u_2}(e)D_1(e)^\top\nu_1(de)\Big\}ds\\
			&+\int_0^t\mathbb{E}\Big[\widetilde{\Big[\tilde{\tilde{\varphi}}^{u_2}\big(\tilde{\tilde{X}}^{u_2}\big)^\top\Big]}\big(\tilde{\tilde{X}}^{u_2}\big)^\top\Big\vert\mathcal{Y}_2^{u_2}\Big]H_2^\top (K^{-1}_2)^\top d\check{U}^c(s),
		\end{aligned}
	\end{equation}
respectively, where
	\begin{equation*}
		\begin{aligned}
			\Xi_3(t)=&\int_t^T\Big\{\Xi_3\big[A^\top-PB_1R_1^{-1}B_1^\top\big]^\top+\big[A^\top-PB_1R_1^{-1}B_1^\top\big]\Xi_3+\Xi_2H_2^\top (K^{-1}_2)^\top K^{-1}_2H_2\Xi_2^\top\\
			&\qquad -\mathbb{E}\big[\phi^{u_2}\big(\phi^{u_2}\big)^\top\big\vert\mathcal{Y}^{u_2}_2\big]-\int_{E_1}\mathbb{E}\big[\psi^{u_2}(e)\big(\psi^{u_2}(e)\big)^\top\big\vert\mathcal{Y}^{u_2}_2\big]\nu_1(de)\Big\}ds\\
			&-\int_t^T\mathbb{E}\Big[\widetilde{\Big[\tilde{\tilde{\varphi}}^{u_2}\big(\tilde{\tilde{\varphi}}^{u_2}\big)^\top\Big]}\big(\tilde{\tilde{X}}^{u_2}\big)^\top
            \Big\vert\mathcal{Y}_2^{u_2}(s)\Big]H_2^\top (K^{-1}_2)^\top d\check{U}^c(s),\\
	    	\Xi_4(t)=&\int_0^t\Big\{\Xi_4\big[A-\big[\Xi_1H_2^\top (K^{-1}_2)^\top+C_2(s)\big]K^{-1}_2H_2\big]^\top+\big[A-B_1R_1^{-1}B_1^\top P(s)\big]\Xi_4\\
			&\qquad -\Xi_6PB_1R_1^{-1}B_1^\top-\Xi_5B_1R_1^{-1}B_1^\top-B_1R_1^{-1}B_1^\top\Xi_2+\Xi_1H_2^\top (K^{-1}_2)^\top C_2^\top\\
			&\qquad +C_1C_1^\top+C_2C_2^\top+\int_{E_1}D_1,e)D_1(e)^\top\nu_1(de)\Big\}ds\\		
            &+\int_0^t\mathbb{E}\Big[\widetilde{\big[\big(\hat{X}^{u_2}-\check{X}^{u_2}\big)\big(\tilde{\tilde{X}}^{u_2}\big)^\top\big]}\big(\tilde{\tilde{X}}^{u_2}\big)^\top\Big\vert\mathcal{Y}_2^{u_2}\Big]
              H_2^\top (K^{-1}_2)^\top d\check{U}^c(s),
		\end{aligned}
	\end{equation*}
	\begin{equation*}
		\begin{aligned}
			\Xi_5(t)=&\int_0^t\Big\{-\Xi_5\big[A^\top-PB_1R_1^{-1}B_1^\top\big]^\top+\big[A-B_1R_1^{-1}B_1^\top P\big]\Xi_5-\Xi_4H_2^\top (K^{-1}_2)^\top K^{-1}_2H_2\Xi_2\\
			&\qquad -B_1R_1^{-1}B_1^\top\Xi_3+C_1(\check{\phi}^{u_2})^\top+\int_{E_1}D_1(e)(\check{\psi}^{u_2}(e))^\top\nu_1(de)\Big\}ds\\		
         	&+\int_0^t\mathbb{E}\Big[\widetilde{\big[\big(\hat{X}^{u_2}-\check{X}^{u_2}\big)\big(\tilde{\tilde{\varphi}}^{u_2}\big)^\top\big]}\big(\tilde{\tilde{X}}^{u_2}\big)^\top
            \Big\vert\mathcal{Y}_2^{u_2}\Big]H_2^\top (K^{-1}_2)^\top d\check{U}^c(s),
		\end{aligned}
	\end{equation*}
	and
	\begin{equation*}
		\begin{aligned}
			\Xi_6(t)=&\int_0^t\Big\{\big[A-B_1R_1^{-1}B_1^\top P(s)\big]\Xi_6+\Xi_6\big[A-B_1R_1^{-1}B_1^\top P\big]^\top-\Xi_5B_1R_1^{-1}B_1^\top-B_1R_1^{-1}B_1^\top\Xi_5^\top\\
			&\qquad -\Xi_4H_2^\top (K^{-1}_2)^\top\big[\Xi_1H_2^\top (K^{-1}_2)^\top+C_2\big]^\top-\big[\Xi_1H_2^\top (K^{-1}_2)^\top+C_2\big]K^{-1}_2H_2\Xi_4^\top\\
			&\qquad +C_1C_1^\top+\big[\Xi_1H_2^\top (K^{-1}_2)^\top+C_2\big]\big[\Xi_1H_2^\top (K^{-1}_2)^\top+C_2\big]^\top\\
			&\qquad +\int_{E_1}D_1(e)D_1(e)^\top\nu_1(de)+\int_{E_2}D_2(e)D_2(e)^\top\nu_2(de)\Big\}ds\\	
     		&+\int_0^t\mathbb{E}\Big[\widetilde{\big[\big(\hat{X}^{u_2}-\check{X}^{u_2}\big)\big(\hat{X}^{u_2}-\check{X}^{u_2}\big)^\top\big]}\big(\tilde{\tilde{X}}^{u_2}\big)^\top\Big\vert\mathcal{Y}_2^{u_2}\Big]
              H_2^\top (K^{-1}_2)^\top d\check{U}^c(s)\\
			&+\int_0^t\int_{E_2}D_2(e)D_2(e)^\top\tilde{N}_2(de,ds).
		\end{aligned}
	\end{equation*}
\end{lemma}

\subsection{The leader's optimal control problem}

In this subsection, we continue to deal with the leader's problem, which is a partially observed SLQ optimal control problem, subject to the state equation \eqref{state equation leader} and the cost functional $J_2(\bar{u}_1,u_2)$ in \eqref{cost J1 J2 Sec2}. Similar as in section \ref{Sec follower's}, by introducing
\begin{equation*}
	\left\{
	\begin{aligned}
		\check{J}_2(u_2)&\coloneqq\check{J}_2(\bar{u}_1,u_2)\coloneqq\mathbb{E}\left[\int_0^T\frac{1}{2}\langle R_2(t)u_2(t),u_2(t)\rangle dt\right] +\mathbb{E}\big[\check{X}^{u_2}(T)\big]+\frac{\theta_2}{2}\text{Var}\big[\check{X}^{u_2}(T)\big],\\
		\tilde{J}_2(u_2)&\coloneqq\tilde{J}_2(\bar{u}_1,u_2)\coloneqq\mathbb{E}\big[\tilde{\tilde{X}}^{u_2}(T)\big]+\frac{\theta_2}{2}\text{Var}\big[\tilde{\tilde{X}}^{u_2}(T)\big],
	\end{aligned}
	\right.
\end{equation*}
we decompose the cost functional $J_2(\bar{u}_1,u_2)$ in \eqref{cost J1 J2 Sec2} as
\begin{equation}\label{decom J2}
	J_2(\bar{u}_1,u_2)=\check{J}_2(\bar{u}_1,u_2)+\tilde{J}_2(\bar{u}_1,u_2)\equiv\check{J}_2(u_2)+\tilde{J}_2(u_2),
\end{equation}
corresponding to the state equations \eqref{state equation leader check X}-\eqref{state equation leader check varphi} and \eqref{equation tilde tilde X}-\eqref{equation tilde tilde variphi}, respectively.

As in the follower's problem, $\check{J}_2(u_2)$ is a mean-variance criteria and $\tilde{J}_2(u_2)$ is independent of $u_2$. Denote the mean-variance problem with respect to $\check{J}_2(u_2)$ by
\begin{equation*}
	\mathbf{P_F(\theta_2):}\quad\mbox{Minimize } \check{J}_2(u_2)\mbox{ over }u_2\in\mathcal{U}^L_{ad},\mbox{ subject to }\eqref{state equation leader check X}\mbox{ and }\eqref{state equation leader check varphi}.
\end{equation*}
Define the optimal control set of problem $P_F(\theta_2)$ as
\begin{equation*}
	\Pi_{P_F(\theta_2)}\coloneqq\bigl\{u_2(\cdot)\big\vert u_2(\cdot) \mbox{ is an optimal control of }P_F(\theta_2)\bigr\}.
\end{equation*}

Introducing the cost functional $\check{\mathcal{J}}_2(u_2)$ by
\begin{equation*}
	\check{\mathcal{J}}_2(u_2)=\mathbb{E}\left[\int_0^T\frac{1}{2}\langle R_2(t)u_2(t),u_2(t)\rangle dt+\lambda_2\check{X}^{u_2}(T)+\frac{\theta_2}{2}|\check{X}^{u_2}(T)|^2\right],
\end{equation*}
we have the following auxiliary problem:
\begin{equation*}
	\begin{aligned}
		\mathbf{A_F(\theta_2,\lambda_2):}\quad\mbox{ Minimize }\check{\mathcal{J}}_2(u_2)\mbox{ over }u_2\in\mathcal{U}^L_{ad},\mbox{ subject to }\eqref{state equation leader check X}\mbox{ and }\eqref{state equation leader check varphi}.
		\end{aligned}
\end{equation*}
Define the optimal control set of problem $A_F(\theta_2,\lambda_2)$ as
\begin{equation*}
	\Pi_{A_F(\theta_2,\lambda_2)}\coloneqq\bigl\{u_2(\cdot)\big\vert u_2(\cdot)\mbox{ is an optimal control of }A_F(\theta_2,\lambda_2)\bigr\}.
\end{equation*}

Similar as Lemma \ref{embedding theorem}, the following result shows the relationship between problems $P_F(\theta_2)$ and $A_F(\theta_2,\lambda_2)$.
\begin{lemma}({\bf Embedding theorem of the leader})\label{lemma emb2}
	For any $\lambda_2>0$, one has
	\begin{equation}
		\Pi_{P_F(\theta_2)}\subset\bigcup_{-\infty<\lambda_2<+\infty}\Pi_{A_F(\theta_2,\lambda_2)}.
	\end{equation}
	Moreover, if $\bar{u}_2\in\Pi_{P_F(\theta_2)}$, then $\bar{u}_2\in\Pi_{A_F(\theta_2,\bar{\lambda}_2)}$ with $\bar{\lambda}_2=1-\theta_2\mathbb{E}[\check{X}^{\bar{u}_2}_T]$, where $\check{X}^{\bar{u}_2}$ is the corresponding trajectory.
\end{lemma}

Now we are going to study $A_F(\theta_2,\lambda_2)$, which is rather difficult than $A_F(\theta_1,\lambda_1)$ in the follower's problem and can not be solved by the technique in \cite{ZL00} directly. For this target, in order to deal with $\check{\mathcal{J}}_2(u_2)$, we follow the method developed in Hu et al. \cite{HJX23}.

Let $\bar{u}_2$ be an optimal control, and let $(\bar{\check{X}},\bar{\check{\varphi}},\bar{\check{\phi}})\coloneqq(\check{X}^{\bar{u}_2},\check{\varphi}^{\bar{u}_2},\check{\phi}^{\bar{u}_2})$ be the corresponding optimal state. Then by the SMP of FBSDEPs (see \O skendal and Sulem \cite{OS09}, or Shi and Wu \cite{SW10}), the optimal control $\bar{u}_2$ satisfies
\begin{equation}\label{maximum condition leader-SMP}
	B_2(t)^{\top}m(t)+B_2(t)^{\top}P(t)h(t)+R_2(t)\bar{u}_2(t)=0,\quad t\in[0,T],
\end{equation}
where process quadruple $(h,m,n,r)\equiv\{h(t)\in\mathbb{R},m(t)\in\mathbb{R}^n,n(t)\in\mathbb{R}^n,r(t,e)\in\mathbb{R}^n;0\leqslant t\leqslant T,e\in E_2\}$ satisfies
\begin{equation}\label{adjoint equation hmn leader}
	\left\{
	\begin{aligned}
		dh(t)&=\Big\{-B_1(t)R_1^{-1}(t)B_1(t)^\top m(t)+\big[A(t)-B_1(t)R_1^{-1}(t)B_1(t)^\top P(t)\big]h(t)\Big\}dt,\\
		-dm(t)&=\big[A(t)^\top-P(t)B_1(t)R_1^{-1}(t)B_1(t)^\top\big]m(t)dt-n(t)d\bar{\check{U}}^c(t)-\int_{E_2}r(t,e)\tilde{N}_2(de,dt),\\
		h(0)&=0,\quad m(T)=\theta_2\bar{\check{X}}(T)+\lambda_2.
	\end{aligned}
	\right.
\end{equation}

Condition \eqref{maximum condition leader-SMP} becomes sufficient since $R_2(\cdot)>0$. Actually, from \cite{OS09} and \cite{SW10}, we have the following result.
\begin{lemma}\label{lemma suffi}
	Let Assumption \ref{assumption state} hold. If there exists an admissible control $\bar{u}_2$ satisfying \eqref{maximum condition leader-SMP}, where $(h,m)$ is defined in \eqref{adjoint equation hmn leader} which is coupled with the optimal state $(\bar{\check{X}},\bar{\check{\phi}},\bar{\check{\phi}})$ of
\begin{equation}\label{optimal state equation leader check X}
	\left\{
	\begin{aligned}
		d\bar{\check{X}}(t)&=\Big\{\big[A(t)-B_1(t)R_1^{-1}(t)B_1(t)^\top P(t)\big]\bar{\check{X}}(t)-B_1(t)R_1^{-1}(t)B_1(t)^\top\bar{\check{\varphi}}(t)\\
		&\qquad -B_2(t)R_2^{-1}(t)B_2(t)^\top\big[ m(t)+P(t)h(t)\big]\Big\}dt\\
		&\quad+\Big\{\Xi_1(t)H_2(t)^\top K_2^{-1}(t)^\top+C_2(t)\Big\}d\bar{\check{U}}^c(t)+\int_{E_2}D_2(t,e)\tilde{N}_2(de,dt),\\
		-d\bar{\check{\varphi}}(t)&=\Big\{\big[A(t)^\top-P(t)B_1(t)R_1^{-1}(t)B_1(t)^\top\big]\bar{\check{\varphi}}(t)\\
		&\qquad -P(t)B_2(t)R_2^{-1}(t)B_2(t)^\top\big[m(t)+P(t)h(t)\big]\Big\}dt-\bar{\check{\phi}}(t)d\bar{\check{U}}^c(t),\\
		\bar{\check{X}}(0)&=x_0,\quad\bar{\check{\varphi}}(T)=\bar{\lambda}_1,
	\end{aligned}
	\right.
\end{equation}
where $\bar{\check{U}}^c(t)\coloneqq\int_0^tK^{-1}_2(s)H_2(s)\tilde{\tilde{X}}^{\bar{u}_2}(s)ds+\int_0^tdW_2(s)$, by its terminal condition, then $\bar{u}_2$ is the unique optimal control for the problem $A_F(\theta_2,\lambda_2)$.
\end{lemma}

Set
\begin{equation}\label{X Y Z}
		\mathbb{X}(\cdot)\coloneqq\begin{bmatrix}\bar{\check{X}}(\cdot)\\h(\cdot)\end{bmatrix},\quad
\mathbb{Y}(\cdot)\coloneqq\begin{bmatrix}m(\cdot)\\\bar{\check{\varphi}}(\cdot)\end{bmatrix},\quad
		\mathbb{Z}(\cdot)\coloneqq\begin{bmatrix}n(\cdot)\\\bar{\check{\phi}}(\cdot)\end{bmatrix},\quad
\mathbb{U}(\cdot,e)\coloneqq\begin{bmatrix}r(\cdot,e)\\0\end{bmatrix},
\end{equation}
then \eqref{adjoint equation hmn leader}-\eqref{optimal state equation leader check X} is written as
\begin{equation}\label{hamiltonian high dim leader}
	\left\{
	\begin{aligned}
		d\mathbb{X}(t)&=\bigl\{\mathbb{A}_1(t)\mathbb{X}(t)+\mathbb{A}_2(t)\mathbb{Y}(t)\bigr\}dt+\mathbb{A}_4(t)d\bar{\check{U}}^c(t)+\int_{E_2}\mathbb{A}_5(t,e)\tilde{N}_2(de,dt),\\
		-d\mathbb{Y}(t)&=\bigl\{\mathbb{A}_6(t)\mathbb{X}(t)+\mathbb{A}_1(t)^\top\mathbb{Y}(t)\bigr\}dt
        -\mathbb{Z}(t)d\bar{\check{U}}^c(t)-\int_{E_2}\mathbb{U}(t,e)\tilde{N}_2(de,dt),\\
		\mathbb{X}(0)&=\begin{bmatrix}x_0\\0\end{bmatrix},\quad\mathbb{Y}(T)=\begin{bmatrix}\theta_2\bar{\check{X}}(T)+\bar{\lambda}_2\\\bar{\lambda}_1\end{bmatrix},
	\end{aligned}
	\right.
\end{equation}
where we have denoted
\begin{equation*}
	\begin{aligned}
		&\mathbb{A}_1(\cdot)\coloneqq\begin{bmatrix}
			A-B_1R_1^{-1}B_1^\top P& -B_2R_2^{-1}B_2^\top P \\
			0& A-B_1R_1^{-1}B_1^\top P
		\end{bmatrix},\quad\mathbb{A}_2(\cdot)\coloneqq\begin{bmatrix}
			-B_2R_2^{-1}B_2^\top& -B_1R_1^{-1}B_1^\top \\
			-B_1R_1^{-1}B_1^\top&0
		\end{bmatrix},\\
		&\mathbb{A}_4(\cdot)\coloneqq\begin{bmatrix}
			\Xi_1H_2^\top (K_2^{-1})^\top+C_2\\
			0
		\end{bmatrix},\quad
        \mathbb{A}_5(\cdot)\coloneqq\begin{bmatrix}
			D_2\\
			0
		\end{bmatrix},\quad
		\mathbb{A}_6(\cdot)\coloneqq\begin{bmatrix}
			0&0  \\
			0&-PB_2R_2^{-1}B_2^\top P
		\end{bmatrix}.
	\end{aligned}
\end{equation*}

In order to obtain the state feedback form of the optimal control, the following result is needed to present the decoupled relation among \eqref{X Y Z}.
\begin{lemma}
	Let Assumption \ref{assumption state} hold. Then $(\mathbb{X}(\cdot),\mathbb{Y}(\cdot),\mathbb{Z}(\cdot),\mathbb{U}(\cdot,\cdot))$ admits the following decoupled relations:
	\begin{equation}\label{relation X and Y}
		\left\{
		\begin{aligned}
			\mathbb{Y}(t)&=\alpha_1(t)\mathbb{X}(t)+\alpha_2(t),\\
			\mathbb{Z}(t)&=\alpha_1(t)\mathbb{A}_4(t),\\
			\mathbb{U}(t,e)&=\alpha_1(t)\mathbb{A}_5(t,e),\quad t\in[0,T],\quad e\in E_2,
		\end{aligned}
		\right.
	\end{equation}
where $(\alpha_1(\cdot),\alpha_2(\cdot))$ is the unique solution to
\begin{equation}\label{equation alpha_1,alpha_2}
	\left\{
	\begin{aligned}
		-d\alpha_1(t)&=\bigl\{\mathbb{A}_1(t)^\top\alpha_1(t)+\alpha_1(t)\mathbb{A}_1(t)+\alpha_1(t)\mathbb{A}_2(t)\alpha_1(t)+\mathbb{A}_6(t)\bigr\}dt,\\
		-d\alpha_2(t)&=\big[\alpha_1(t)\mathbb{A}_2(t)+\mathbb{A}_1(t)^{\top}\big]\alpha_2(t)dt,\\
		\alpha_1(T)&=\begin{bmatrix}
			\theta_2&0\\
			0&0
		\end{bmatrix},\quad\alpha_2(T)=\begin{bmatrix}\bar{\lambda}_2\\\bar{\lambda}_1\end{bmatrix}.
	\end{aligned}
	\right.
\end{equation}
\end{lemma}

\begin{proof}
	According to classical Riccati equation theory in Wonham \cite{Wonham68}, \eqref{equation alpha_1,alpha_2} admits a unique solution $(\alpha_1(\cdot),\alpha_2(\cdot))$. Then the decoupled relation \eqref{relation X and Y} follows easily from It\^o's formula.
\end{proof}

Set
\begin{equation*}
	\begin{aligned}
		&\alpha_1\coloneqq\begin{bmatrix}
			\alpha_{1,1}&\alpha_{1,2}\\
			\alpha_{1,3}&\alpha_{1,4}
		\end{bmatrix},\quad\alpha_2\coloneqq\begin{bmatrix}
		\alpha_{2,1}\\
		\alpha_{2,2}
	\end{bmatrix},
	\end{aligned}
\end{equation*}
where $\alpha_{1,1}$, $\alpha_{1,2}$, $\alpha_{1,3}$ and $\alpha_{1,4}$ are $\mathbb{R}^{n\times n}$-valued, $\alpha_{2,1}$ and $\alpha_{2,2}$ are $\mathbb{R}^n$-valued. Then we have the following result.
\begin{theorem}
	Let Assumption \ref{assumption state} hold. Let $(\alpha_1(\cdot),\alpha_2(\cdot))$ be the solution of \eqref{equation alpha_1,alpha_2}. Then problem $A_L(\theta_2,\lambda_2)$ admits a unique optimal control
\begin{equation}\label{equation bar u_2 Theorem}
	\bar{u}_2(t)=-R_2^{-1}(t)B_2(t)^\top\Big\{\alpha_{1,1}(t)\check{X}^*(t)+\big[\alpha_{1,2}(t)+P(t)\big]h^*(t)+\alpha_{2,1}(t)\Big\},
\end{equation}
where $\mathbb{X}^*(\cdot)\coloneqq\begin{bmatrix}\check{X}^*(\cdot)\\h^*(\cdot)\end{bmatrix}$ is the solution of
\begin{equation}\label{equation X*}
	\left\{
	\begin{aligned}
		d\mathbb{X}^*(t)&=\bigl\{\big[\mathbb{A}_1(t)+\mathbb{A}_2(t)\alpha_1(t)\big]\mathbb{X}^*(t)+\mathbb{A}_2(t)\alpha_2(t)\bigr\}dt\\
		&\qquad +\mathbb{A}_4(t)d\bar{\check{U}}^c(t)+\int_{E_2}\mathbb{A}_5(t,e)\tilde{N}_2(de,dt),\\
		\mathbb{X}^*(0)&=\begin{bmatrix}x_0\\0\end{bmatrix},
	\end{aligned}
	\right.
\end{equation}
and
\begin{equation}\label{relations leader}
\left\{
	\begin{aligned}
		\bar{\check{X}}(t)&=\check{X}^*(t),\quad h(t)=h^*(t),\quad m(t)=\alpha_{1,1}(t)\check{X}^*(t)+\alpha_{1,2}h^*(t)+\alpha_{2,1}(t),\\
		\bar{\check{\varphi}}(t)&=\alpha_{1,3}(t)\check{X}^*(t)+\alpha_{1,3}h^*(t)+\alpha_{2,2}(t),\\
        n(t)&=\alpha_{1,1}(t)\big[\Xi_1(t)H_2(t)^\top K_2^{-1}(t)^\top+C_2(t)\big],\\
		\bar{\check{\phi}}(t)&=\alpha_{1,3}(t)\big[\Xi_1(t)H_2(t)^\top K_2^{-1}(t)^\top+C_2(t)\big],\quad r(t,e)=\alpha_{1,1}(t)D_2(t,e).
	\end{aligned}
\right.
\end{equation}
\end{theorem}

\begin{proof}
	Substituting $\bar{u}_2$ defined in \eqref{equation bar u_2 Theorem} into \eqref{state equation leader check X} and \eqref{state equation leader check varphi}, it is easy to check that $(\bar{\check{X}}(\cdot),h(\cdot),\bar{\check{\varphi}}(\cdot),\\
	\bar{\check{\phi}}(\cdot),m(\cdot),n(\cdot),r(\cdot,\cdot))$ defined in \eqref{relations leader} solves  \eqref{state equation leader check X}, \eqref{state equation leader check varphi} and \eqref{adjoint equation hmn leader}. Thus $\bar{u}_2$ satisfies \eqref{maximum condition leader-SMP}. By Lemma \ref{lemma suffi} and Lemma \ref{lemma emb2}, $\bar{u}_2$ is the unique optimal control of $A_L(\theta_2,\lambda_2)$, as well as $P_F(\theta_2)$. From \eqref{decom J2}, noting that $\tilde{J}_2$, $\tilde{\tilde{X}}$ and $\tilde{\tilde{\varphi}}$ are independent of $u_2$, we obtain that $\bar{u}_2$ is the optimal control of the leader's problem.
\end{proof}

\section{Concluding remarks}\label{Sec Conclu}

In this paper, we have studied a Stackelberg differential game for partially observed linear stochastic system with mean-variance criteria, where randomness comes from Brownian motions and Poisson random measures. Following the orthogonal decomposition technique developed in \cite{SX23} and the embedding method in \cite{ZL00}, the original problems were transformed into several fully observed SLQ problems. State feedback representation of open-loop Stgackelberg equilibria was obtained via some Riccati equations. The orthogonal decomposition technique heavily depends on the linear and non-linear filtering with Brownian motions and Poisson random measures, which we have investigated in Appendix \ref{Sec Filter}.

\appendix

\setcounter{equation}{0}
\renewcommand{\theequation}{A.\arabic{equation}}

\section{Appendix: Non-linear filtering}\label{Sec Filter}

In this appendix section, we investigate the linear and non-linear stochastic filtering with Poisson random measure, which will play a crucial role in the main parts of this paper. On a given filtered probability space $(\Omega,\mathcal{F},\mathbb{F},\mathbb{P})$, $W_1$ and $W_2$ are two standard independent Brownian motions, with values in $\mathbb{R}^d$ and $\mathbb{R}^k$, respectively. $\tilde{N}_1(de,dt)$ and $\tilde{N}_2(de,dt)$ are standard 1-dimensional Poisson martingale measures. Some similar definitions and notations are omitted as in Section \ref{Sec Problem formulation}.

Let us consider the following partially observed stochastic system:
\begin{equation}\label{SDEP appendix}
	\left\{
	\begin{aligned}
		dx(t)&=a(t)dt+\sum_{i=1}^2\sigma_i(t)dW_i(t)+\sum_{i=1}^2\int_{E_i}\varpi_i(t,e)\tilde{N}_i(de,dt),\\
		dY(t)&=H(t)dt+K(t)dW_1(t)+\int_{E_1}c(t,e)\tilde{N}_1(de,dt),\\
		x(0)&=x_0,\quad Y(0)=0,
	\end{aligned}
	\right.
\end{equation}
where $\{x(t)\in\mathbb{R}^n;0\leqslant t\leqslant T\}$ is the $\mathbb{R}^n$-valued unobservable signal process, and $\{Y(t)\in\mathbb{R};0\leqslant t\leqslant T\}$ is the observable process with values in $\mathbb{R}$.

Let $\mathcal{Y}_t\coloneqq\sigma(Y(s);0\leqslant s\leqslant t),0\leqslant t\leqslant T$. The problem of optimal filtering for a partially observable system is to derive the equation satisfied by the non-linear filtering prosess $\mathbb{E}[x(t)\vert\mathcal{Y}_t]$. We need the following assumption.
\begin{assumption}\label{assumption nonlinear filtering}
	\begin{enumerate}[\bfseries (1)]
		\item $a:[0,T]\times\Omega\to\mathbb{R}^n$, $\sigma_i:[0,T]\times\Omega\to\mathbb{R}^n$, $i=1,2$, $\varpi_i:E_i\times[0,T]\times\Omega\to\mathbb{R}^n$, $i=1,2$,  $H:[0,T]\times\Omega\to\mathbb{R}$, $K:[0,T]\to\mathbb{R}$ and $c:E_1\times[0,T]\to\mathbb{R}$ are measurable, $\mathbb{F}$-adapted, and
		\begin{equation*}
			\begin{aligned}
				&\mathbb{E}\int_0^T\left\{|a(t)|+\sum_{i=1}^2|\sigma_i(t)|^2+\sum_{i=1}^2|\varpi_1(t,e)|^2\nu_i(de)\right.\\
				&\qquad\quad +|H(t)|+|K(t)|^2+|c(t,e)|^2\nu_1(de)\Bigg\}dt<\infty;
			\end{aligned}
		\end{equation*}
	\item $K(t)$ is invertible and $K^{-1}(t)$, is bounded on $[0,T]$; $c(t,e)$ is invertible and $c^{-1}(t,e)$, is bounded on $[0,T]\times E_1$.
	\end{enumerate}
\end{assumption}

\begin{proposition}%%%%%%%这个命题来自于哪个文献？
	Let $x\equiv\{x(t);0\leqslant t\leqslant T\}$ be a measurable process. There exists a unique (up to indistinguishability) optional process $\hat{x}\equiv\{\hat{x}(t);0\leqslant t\leqslant T\}$ such that, for all stopping time $\tau$,
	\begin{equation*}
		\hat{x}(\tau)\mathbbm{1}_{(\tau<\infty)}=\mathbb{E}[x(\tau)\mathbbm{1}_{(\tau<\infty)}\vert\mathcal{Y}_{\tau}],\quad\mathbb{P}\mbox{-}a.s.
	\end{equation*}
\end{proposition}

\begin{remark}
	We remark that $\hat{x}(t)=\mathbb{E}[x(t)\vert\mathcal{Y}_t]$, $\mathbb{P}$-a.s., for each fixed time $t$, but the null set depends in $t$. Therefore were we simply to write as a process $\{\mathbb{E}[x(t)\vert\mathcal{Y}_t]; t\geqslant0\}$ instead of $\hat{x}$, it would not be uniquely determined almost surely. This is why we use the optional projection. Hereafter, we shall use the notation $\hat{x}\equiv\{\hat{x}(t);0\leqslant t\leqslant T\}$ for the optional projection of a process $x\equiv\{x(t);0\leqslant t\leqslant T\}$, and use $\tilde{x}\equiv\{\tilde{x}(t); 0\leqslant t\leqslant T\}$ for the difference $x-\hat{x}$.
\end{remark}

Now we are going to  establish the filtering equation for $\hat{x}(t)=\mathbb{E}[x(t)\vert\mathcal{Y}_t]$, where $x$ satisfies the SDEP \eqref{SDEP appendix}. Let $V\equiv\{V(t);0\leqslant t\leqslant T\}$ be the innovation process defined by
\begin{equation}\label{innovation}
		V(t)\coloneqq Y(t)-\int_0^t\hat{H}(s)ds=V^c(t)+V^d(t),
\end{equation}
where $V^c(t)\coloneqq\int_0^t\tilde{H}(s)ds+\int_0^tK(s)dW_1(s)$ is the continuous part of the innovation process $V$, and $V^d(t)\coloneqq\int_0^t\int_{E_1}c(s,e)\tilde{N}_1(de,ds)$ is its pure jump part.

First we have the following result.
\begin{lemma}\label{lemma check V 3.1}
	Let Assumption \ref{assumption nonlinear filtering} hold. The process $\check{V}^c\equiv\{\check{V}^c(t);0\leqslant t\leqslant T\}$, defined by
\begin{equation}\label{innovation continuous part}
	\check{V}^c(t)\coloneqq\int_0^tK^{-1}(s)dV^c(s),
\end{equation}
is a standard $\{\mathcal{Y}_t\}$-Brownian motion in $\mathbb{R}^d$.
\end{lemma}

\begin{proof}
	Evidently, $\check{V}^c$ is a continuous, $\{\mathcal{Y}_t\}$-adapted, integrable process. Additionally, we have
\begin{equation*}
		d\check{V}^c(t)=K^{-1}(t)dV^c(t)=K^{-1}(t)\tilde{H}(t)dt+dW_1(t).
\end{equation*}
Thus, for $0\leqslant s<t\leqslant T$, we obtain
\begin{equation*}
	\check{V}^c(t)-\check{V}^c(s)=W_1(t)-W_1(s)+\int_s^tK^{-1}(r)\tilde{H}(r)dr.
\end{equation*}

Given that for $r\geqslant s$, $\tilde{H}(r)$ is independent of $\mathcal{Y}_s$,
\begin{equation*}
	\begin{aligned} \mathbb{E}\big[\check{V}^c(t)-\check{V}^c(s)\big\vert\mathcal{Y}_s\big]
    &=\mathbb{E}\Big[\mathbb{E}\big[W_1(t)-W_1(s)\big\vert\mathcal{F}_s\big]\Big\vert\mathcal{Y}_s\Big]+\int_s^tK^{-1}(r)\mathbb{E}\big[\tilde{H}(r)\big\vert\mathcal{Y}_s\big]dr\\
	&=\int_{s}^{t}K^{-1}(r)\mathbb{E}\big[\tilde{H}(r)\big]dr=0.
	\end{aligned}
\end{equation*}
This indicates that $\check{V}^c$ is a $\{\mathcal{Y}_t\}$-martingale. Further, it follows from It\^o's formula that
\begin{equation*}
	\begin{aligned}
		d\big[\check{V}^c(t)\check{V}^c(t)^\top\big]&=\check{V}^c(t)\big[K^{-1}(t)\tilde{H}(t)\big]^\top dt+\check{V}^c(t)dW_1(t)^\top\\
		&\quad +K^{-1}(t)\tilde{H}(t)\check{V}^c(t)^\top dt+[dW_1(t)]\check{V}^c(t)^\top+\mathbbm{1}_ddt,
	\end{aligned}
\end{equation*}
and hence
\begin{equation}\label{equation A4}
	\begin{aligned}
		&\check{V}^c(t)\check{V}^c(t)^\top-\check{V}^c(s)\check{V}^c(s)^\top=\left[\int_s^tK^{-1}(r)\tilde{H}(r)\check{V}^c(r)^\top dr\right]^\top+\int_s^t\check{V}^c(r)dW_1(r)^\top\\
		&\quad +\int_s^tK^{-1}(r)\tilde{H}(r)\check{V}^c(r)^\top dr+\left[\int_s^t\check{V}^c(r)dW_1(r)^\top\right]^\top+(t-s)\mathbbm{1}_d.
	\end{aligned}
\end{equation}
Since $\tilde{H}(r)$ is independent of $\mathcal{Y}_s$, for $s\leqslant r$, we have
\begin{equation*}
	\begin{aligned}
		&\mathbb{E}\left[\int_s^tK^{-1}(r)\tilde{H}(r)\check{V}^c(r)^\top dr\Big\vert\mathcal{Y}_s\right]=\int_s^tK^{-1}(r)\mathbb{E}\Big[\tilde{H}(r)\check{V}^c(r)^\top\Big\vert\mathcal{Y}_s\Big]dr\\
		&=\int_s^tK^{-1}(r)\mathbb{E}\Big[\mathbb{E}\Big[\tilde{H}(r)\check{V}^c(r)^\top\big\vert\mathcal{Y}_r\Big]\Big\vert\mathcal{Y}_s\Big]dr
		=\int_s^tK^{-1}(r)\mathbb{E}\Big[\mathbb{E}\big[\tilde{H}(r)\big]\check{V}^c(r)^\top\big\vert\mathcal{Y}_s\Big]dr=0.
	\end{aligned}
\end{equation*}
For the stochastic integral part, we have
\begin{equation*}
	\begin{aligned}
		\mathbb{E}\left[\int_s^t\check{V}^c(t)dW_1(r)^\top\Big\vert\mathcal{Y}_s\right]=\mathbb{E}\left[\mathbb{E}\left[\int_s^t\check{V}^c(r)dW_1(r)^\top\Big\vert\mathcal{F}_s\right]\Big\vert\mathcal{Y}_s\right]=0.
	\end{aligned}
\end{equation*}
By taking conditional expectations with respect to $\mathcal{Y}_s$ on both sides of \eqref{equation A4}, we get
\begin{equation*}
	\mathbb{E}\Big[\check{V}^c(t)\check{V}^c(t)^\top-\check{V}^c(s)\check{V}^c(s)^\top\big\vert\mathcal{Y}_s\Big]=(t-s)\mathbbm{1}_d.
\end{equation*}
This implies that the cross-variations are given by
\begin{equation*}
	\langle\check{V}^c_i,\check{V}^c_j\rangle_t=\delta_{ij}t;\qquad1\leqslant i,j\leqslant d.
\end{equation*}
Based on the martingale characterization of Brownian motion, it can be concluded that $\check{V}^c$ is a standard $d$-dimensional Brownian motion.
\end{proof}

Consider the following process
\begin{equation}\label{equation Lambda}
	\Lambda(t)\coloneqq\hat{x}(t)-\hat{x}_0-\int_{0}^{t}\hat{a}(s)ds.
\end{equation}
Ii is clearly that $\Lambda\equiv\{\Lambda(t);0\leqslant t\leqslant T\}$ is $\{\mathcal{Y}_t\}$-adapted and c\`adl\`ag, with $\Lambda(0)=0$ almost surely. We are going to prove that $\Lambda$ is a $\mathcal{Y}_t$-martingale.
\begin{lemma}
	The process $\Lambda=\{\Lambda(t);0\leqslant t\leqslant T\}$ defined by \eqref{equation Lambda} is a c\`adl\`ag, square-integrable $\mathcal{Y}_t$-martingale with $\Lambda(0)=0$, a.s.
\end{lemma}

\begin{proof}
	The square-integrability of $\Lambda$ follows from the square-integrability of $x$. For fixed $0\leqslant s<t\leqslant T$, we have
	\begin{equation*}
		\begin{aligned}
			&\mathbb{E}\big[\Lambda(t)-\Lambda(s)\big\vert\mathcal{Y}_s\big]
			=\mathbb{E}\big[x(t)-x(s)\big\vert\mathcal{Y}_s\big]-\int_s^t\mathbb{E}\big[\hat{a}(r)\big\vert\mathcal{Y}_s\big]dr\\
			&=\mathbb{E}\big[x(t)-x(s)\big\vert\mathcal{Y}_s\big]-\mathbb{E}\Big[\int_s^ta(r)dr\big\vert\mathcal{Y}_s\Big]\\
			&=\mathbb{E}\left[\sum_{i=1}^2\int_s^t\sigma_i(r)dW_i(r)+\sum_{i=1}^{2}\int_{s}^{t}\int_{E_i}\varpi_i(r,e)\tilde{N}_i(de,dr)\bigg\vert\mathcal{Y}_s\right]\\
			&=\mathbb{E}\left[\mathbb{E}\left[\sum_{i=1}^2\int_s^t\sigma_i(r)dW_i(r)+\sum_{i=1}^2\int_s^t\int_{E_i}\varpi_i(r,e)\tilde{N}_i(de,dr)\bigg\vert\mathcal{F}_s\right]\bigg\vert\mathcal{Y}_s\right]=0.
		\end{aligned}
	\end{equation*}
This shows that $\Lambda$ is a $\mathcal{Y}_t$-martingale.
\end{proof}		
								
It should be noted that for the general non-linear filtering problem, we can only obtain $\sigma\{V(s);0\leqslant s\leqslant t\}\subset\mathcal{Y}_t$. There remains uncertainty regarding whether the process $\Lambda$ is adapted to the smaller filtration generated by the innovation process $V$, as well as $\check{V}^c$ and $\check{V}^d$. Consequently, we are unable to directly apply the martingale representation theorem to deduce that $\Lambda$ can be expressed as a stochastic integral with respect to $\check{V}^c$ and $\check{V}^d$. Fortunately, thanks to Theorem 204 and Lemma 209 of Situ \cite{Situ05}, there exist a square-integrable, $\{\mathcal{Y}_t\}$-predictable process $k_1\equiv\{k_1(t)\in\mathbb{R}^n;0\leqslant t\leqslant T\}$ and a square-integrable, $\{\mathcal{Y}_t\}$-predictable process $k_2\equiv\{k_2(t,e)\in\mathbb{R}^n;0\leqslant t\leqslant T,e\in E\}$ such that
\begin{equation}\label{Situ Rong}
	\Lambda(t)=\int_0^tk_1(s)d\check{V}^c(s)+\int_0^t\int_{E_1}k_2(s,e)\tilde{N}_1(de,ds).
\end{equation}
We are going to determine $k_1$ and $k_2$.

\begin{proposition}
	There exist $\mathbb{R}^n$- and $\mathbb{R}^n$-valued, square-integrable, $\{\mathcal{Y}_t\}$-predictable processes $k_1\equiv\{k_1(t);0\leqslant t\leqslant T\}$ and $k_2\equiv\{k_2(t,e);0\leqslant t\leqslant T,e\in E\}$, respectively, such that \eqref{Situ Rong} holds. Further, $k_1$ and $k_2$ are given by
\begin{equation}\label{k1 and k2}
	\left\{
	\begin{aligned}
		&k_1(t)=\mathbb{E}\big[\tilde{x}(t)\tilde{H}(t)^\top\big\vert\mathcal{Y}_t\big]K^{-1}(t)^\top+\mathbb{E}[\sigma_1(t)\vert\mathcal{Y}_t],\\
		&k_2(t,e)=\mathbb{E}[\varpi_1(t,e)\vert\mathcal{Y}_t].
	\end{aligned}
	\right.
\end{equation}
\end{proposition}

\begin{proof}
	The existence belongs to Theorem 204 and Lemma 209 of \cite{Situ05}. Next, we are going to determine $k_1$ and $k_2$.
	
	Let $\varrho_1\equiv\{\varrho_1(t);0\leqslant t\leqslant T\}$ and $\varrho_2\equiv\{\varrho_2(t,e);0\leqslant t\leqslant T,e\in E\}$ be fixed but arbitrary $\mathbb{R}^n$- and $\mathbb{R}^n$-valued, square-integrable, $\{\mathcal{Y}_t\}$-predictable processes, respectively. Define the $\{\mathcal{Y}_t\}$-martingale
	\begin{equation*}
		\Upsilon(t)\coloneqq\int_0^t\varrho_1(s)d\check{V}^c(s)+\int_0^t\int_{E_1}\varrho_2(s,e)\tilde{N}_1(de,ds).
	\end{equation*}
By using It\^o's formula, we have, for all $t\in[0,T]$,
\begin{equation}\label{equation Lambda Upsilon 1}
	\mathbb{E}\big[\Lambda(t)\Upsilon(t)^\top\big]=\mathbb{E}\left[\int_0^tk_1(s)\varrho_1(s)^\top ds+\int_0^t\int_{E_1}k_2(s,e)\varrho_2(s,e)^\top\nu_1(de)ds\right].
\end{equation}
On the other hand, noting \eqref{equation Lambda}, we have
\begin{equation*}
	\mathbb{E}\big[\Lambda(t)\Upsilon(t)^\top\big]=\mathbb{E}\big[\hat{x}(t)\Upsilon(t)^\top\big]-\int_{0}^{t}\mathbb{E}\big[\hat{a}(s)\Upsilon(t)^\top\big]ds.
\end{equation*}
It follows from the martingale property of $\Upsilon\equiv\{\Upsilon(t), 0\leqslant t\leqslant T\}$ that, for $0\leqslant s\leqslant t\leqslant T$,
\begin{equation*}
	\begin{aligned}
		&\mathbb{E}\big[\hat{a}(s)\Upsilon(t)\top\big]=\mathbb{E}\big[\mathbb{E}\big[\hat{a}(s)\Upsilon(t)^\top\big]\big]=\mathbb{E}\big[\hat{a}(s)\mathbb{E}\big[\Upsilon(t)^\top\big\vert\mathcal{Y}_s\big]\big]\\
		&=\mathbb{E}\big[\hat{a}(s)\Upsilon(s)^\top\big]=\mathbb{E}\big[\mathbb{E}\big[a(s)\big\vert\mathcal{Y}_s\big]\Upsilon(s)^\top\big]=\mathbb{E}\big[a(s)\Upsilon(s)^\top\big].
	\end{aligned}
\end{equation*}
It follows that
\begin{equation}\label{equation Lambda Upsilon}
	\mathbb{E}\big[\Lambda(t)\Upsilon(t)^\top\big]=\mathbb{E}\big[x(t)\Upsilon(t)^\top\big]-\mathbb{E}\left[\int_0^ta(s)\Upsilon(s)^\top ds\right].
\end{equation}
Note that
\begin{equation*}
	\begin{aligned}
		d\Upsilon(t)&=\varrho_1(t)d\check{V}^c(t)+\int_{E_1}\varrho_2(t,e)\tilde{N}_1(de,dt)\\
		&=\varrho_1(t)K^{-1}(t)\tilde{H}(t)dt+\varrho_1(t)dW_1(t)+\int_{E_1}\varrho_2(t,e)\tilde{N}_1(de,dt).
	\end{aligned}
\end{equation*}
Thus, integration by parts yields
\begin{equation}\label{equation x Upsilon}
	\begin{aligned}
		\mathbb{E}\big[x(t)\Upsilon(t)^\top\big]&=\mathbb{E}\left[\int_0^ta(s)\Upsilon(s)^\top ds\right]+\mathbb{E}\left[\int_0^tx(s)\tilde{H}(s)^\top K^{-1}(s)^\top\varrho_1(s)^\top ds\right]\\
		&\quad +\mathbb{E}\left[\int_0^t\sigma_1(s)\varrho_1(s)^\top ds\right]+\mathbb{E}\left[\int_0^t\int_{E_1}\varpi_1(s,e)\varrho_2(s,e)^\top\nu_2(de)ds\right].
	\end{aligned}
\end{equation}
Noting that $\varrho_1,\varrho_2$ are $\mathcal{Y}_s$-adapted, $\tilde{H}(s)$ is independent of $\mathcal{Y}_s$ and $\mathbb{E}[\tilde{H}(s)]=0$, we can deduce
\begin{equation*}
	\begin{aligned}
		&\mathbb{E}\left[\int_0^tx(s)\tilde{H}(s)^\top K^{-1}(s)^\top\varrho_1(s)^\top ds\right]\\
		&=\mathbb{E}\left[\int_0^t\hat{x}(s)\tilde{H}(s)^\top K^{-1}(s)^\top\varrho_1(s)^\top ds\right]+\mathbb{E}\left[\int_0^t\tilde{x}(s)\tilde{H}(s)^\top K^{-1}(s)^\top\varrho_1(s)^\top ds\right]\\
        &=\mathbb{E}\left[\int_0^t\hat{x}(s)\mathbb{E}\big[\tilde{H}(s)^\top\big\vert\mathcal{Y}_s\big]K^{-1}(s)^\top\varrho_1(s)^\top ds\right] +\mathbb{E}\left[\int_0^t\mathbb{E}\big[\tilde{x}(s)\tilde{H}(s)^\top\big\vert\mathcal{Y}_s\big]K^{-1}(s)^\top\varrho_1(s)^\top ds\right]\\
		&=\mathbb{E}\left[\int_0^t\mathbb{E}\big[\tilde{x}(s)\tilde{H}(s)^\top\big\vert\mathcal{Y}_s\big]K^{-1}(s)^\top\varrho_1(s)^\top ds\right].
	\end{aligned}
\end{equation*}
Hence, we can draw from \eqref{equation x Upsilon} that
\begin{equation*}
	\begin{aligned}
        \mathbb{E}\big[x(t)\Upsilon(t)^\top\big]&=\mathbb{E}\left[\int_0^ta(s)\Upsilon(s)^\top ds\right] +\mathbb{E}\left[\int_0^t\mathbb{E}\big[\tilde{x}(s)\tilde{H}(s)^\top\big\vert\mathcal{Y}_s\big]K^{-1}(s)^\top\varrho_1(s)^\top ds\right]\\
		&\quad +\mathbb{E}\left[\int_0^t\sigma_1(s)\varrho_1(s)^\top ds\right]+\mathbb{E}\left[\int_0^t\int_{E_1}\varpi_1(s,e)\varrho_2(s,e)^\top\nu_2(de)ds\right],
	\end{aligned}
\end{equation*}
which, combined with \eqref{equation Lambda Upsilon}, gives
\begin{equation}\label{equation Lambda Upsilon 3}
	\begin{aligned}
		\mathbb{E}\big[\Lambda(t)\Upsilon(t)^\top\big]&=\mathbb{E}\left[\int_0^t\mathbb{E}\big[\tilde{x}(s)\tilde{H}(s)^\top\big\vert\mathcal{Y}_s\big]K^{-1}(s)^\top\varrho_1(s)^\top ds\right]\\
		&\quad +\mathbb{E}\left[\int_0^t\sigma_1(s)\varrho_1(s)^\top ds\right]+\mathbb{E}\left[\int_0^t\int_{E_1}\varpi_1(s,e)\varrho_2(s,e)^\top\nu_2(de)ds\right].
	\end{aligned}
\end{equation}
Comparing \eqref{equation Lambda Upsilon 1} and \eqref{equation Lambda Upsilon 3} and noting that $\varrho_1$ and $\varrho_2$ are arbitrary, we deduce that
\begin{equation*}
	\left\{
	\begin{aligned}
		&k_1(t)=\mathbb{E}\big[\tilde{x}(t)\tilde{H}(t)^\top\big\vert\mathcal{Y}_t\big]K^{-1}(t)^\top+\mathbb{E}\big[\sigma_1(t)\big\vert\mathcal{Y}_t\big],\\
		&k_2(t,e)=\mathbb{E}\big[\varpi_1(t,e)\big\vert\mathcal{Y}_t\big].
	\end{aligned}
	\right.
\end{equation*}
The conclusion now follows.
\end{proof}

The filtering equation of $\hat{x}$ follows immediately.

\begin{lemma}\label{lemma nonlinear filtering equation}
	Let Assumption \ref{assumption nonlinear filtering} hold. The optimal filter $\hat{x}(t)=\mathbb{E}\big[x(t)\big\vert\mathcal{Y}_t\big]$ satisfies the following non-linear filtering equation:
	\begin{equation}\label{equation non-linear filtering equation}
		\left\{
		\begin{aligned}
			d\hat{x}(t)&=\hat{a}(t)dt+\left\{\mathbb{E}\big[\tilde{x}(t)\tilde{H}(t)^\top\big\vert\mathcal{Y}_t\big]K^{-1}(t)^\top+\hat{\sigma}_1(t)\right\}d\check{V}^c(t)\\
                      &\quad +\int_{E_1}\hat{\varpi}_1(t,e)\tilde{N}_1(de,dt),\\
			\hat{x}(0)&=x_0.
		\end{aligned}
		\right.
	\end{equation}
\end{lemma}


\begin{thebibliography}{99}

\bibitem{CS19} L. Chen, Y. Shen, Stochastic Stackelberg differential reinsurance games under time-inconsistent mean-variance framework. \emph{Insur. Math. Econom.}, 88, 120-137, 2019.

\bibitem{GLS24} G.H. Guan, Z.X. Liang, and Y.L. Song, A Stackelberg reinsurance-investment game under $\alpha$-maxmin mean-variance criterion and stochastic volatility. \emph{Scand. Actuar. J.}, 2024(1), 28-63, 2024.

\bibitem{HJX23} M.S. Hu, S.L. Ji, and X.L. Xue, Optimization under rational expectations: a framework of fully coupled forward-backward stochastic linear quadratic systems. \emph{Math. Oper. Res.}, 48(3), 1767-1790, 2023.

\bibitem{HS12} J.H. Huang, J.T. Shi, Maximum principle for optimal control of fully coupled forward-backward stochastic differential delayed equations. \emph{ESAIM Control Optim. Calc. Var.}, 18(4), 1073-1096, 2012.

\bibitem{HZ25} Y.-J. Huang, S.H. Zhu, Mean-variance Stackelberg games with asymmetric information. \emph{arXiv:2509.03669v1}

\bibitem{LSZ24} Y.C. Li, Y.Z. Hu, J.T. Shi, and Y.Y. Zheng, A linear-quadratic partially observed Stackelberg stochastic differential game with multiple followers and its application to multi-agent formation control. \emph{arXiv:2412.07159v2}

\bibitem{LY22} D.P. Li, V.R. Young, Stackelberg differential game for reinsurance: Mean-variance framework and random horizon. \emph{Insur. Math. Econom.}, 102, 42-55, 2022.

\bibitem{Lin21} Y.N. Lin, Linear quadratic open-loop Stackelberg game for stochastic systems with Poisson jumps. \emph{J. Frankl. Inst.}, 358(10), 5262-5280, 2021.

\bibitem{LS25} J.T. Lin, J.T. Shi, A risk-sensitive global maximum principle for controlled fully coupled FBSDEs with applications. \emph{Math. Control Relat. Fields}, 15(1), 179-205, 2025.

\bibitem{LS25+} J.T. Lin, J.T. Shi, Global maximum principle for partially observed risk-sensitive progressive optimal control of FBSDE with Poisson jumps. \emph{Stoch. Proc. Appl.}, 195, 104870, 2026.

\bibitem{LL25} Z.X. Liang, X.D. Luo, Stackelberg reinsurance and premium decisions with MV criterion and irreversibility. \emph{SIAM J. Finan. Math.}, 16(1), 167-199, 2025.

\bibitem{LY25} X.Q. Liang, V.R. Young, Two Stackelberg games in life insurance: Mean-variance criterion. \emph{Astin Bull.}, 55(1), 178-203, 2025.

\bibitem{Meng09} Q.X. Meng, A maximum principle for optimal control problem of fully coupled forward-backward stochastic systems with partial information. \emph{Sci. China Math.}, 52(7), 1579-1588, 2009.

\bibitem{Moon21} J. Moon, Linear-quadratic stochastic Stackelberg differential games for jump-diffusion systems. \emph{SIAM J. Control Optim.}, 59(2), 954-976, 2021.

\bibitem{Moon25} J. Moon, Linear-quadratic stochastic Stackelberg differential games with asymmetric information for systems driven by multi-dimensional jump-diffusion processes. \emph{J. Math. Anal. Appl.}, 544(2), 129068, 2025.

\bibitem{MB24} J. Moon, T. Ba\c{s}ar, Separation principle for partially-observed linear-quadratic optimal control for mean-field type stochastic systems. \emph{IEEE Trans. Automat. Control.}, 69(12), 8370-8385, 2024.

\bibitem{MDB19} J. Moon, T.E. Ducan, and T. Ba\c{s}ar, Risk-sensitive zero-sum differential games. \emph{IEEE Trans. Autom. Control}, 64(4), 1503-1518, 2019.

\bibitem{OS09} B. \O ksendal, A. Sulem, Maximum principles for optimal control of forward-backward stochastic differential equations with jumps. \emph{SIAM J. Control Optim.}, 48(5), 2945-2976, 2009.

\bibitem{SWX16} J.T. Shi, G.C. Wang, and J. Xiong, Leader-follower stochastic differential game with asymmetric information and applications. \emph{Automatica}, 63, 60-73, 2016.

\bibitem{SWX20} J.T. Shi, G.C. Wang, and J. Xiong, Stochastic linear quadratic Stackelberg differential game with overlapping information. \emph{ESAIM Control Optim. Calc. Var.}, 26, 83, 2020.

\bibitem{SW10} J.T. Shi, Z. Wu, Maximum principle for forward-backward stochastic control system with random jumps and applications to finance. \emph{J. Syst. Sci. Complex.}, 23(2), 219-231, 2010.

\bibitem{Situ05} R. Situ, \emph{Theory of Stochastic Differential Equations with Jumps and Applications}, Springer-Verlag, New York, 2005.

\bibitem{Stackelberg34} H. von Stackelberg, {\it Marktform Und Gleichgewicht}, Springer, Vinenna, 1934. (Translated by W. Engels, {\it Market Structure and Equilibrium}, Springer, Berlin, 2010.)

\bibitem{SX23} J.R. Sun, J. Xiong, Stochastic linear-quadratic optimal control with partial observation. \emph{SIAM J. Control Optim.}, 61(3), 1231-1247, 2023.

\bibitem{WLZ25} F.D. Wang, Z.B. Liang, and C.B. Zhang, Time-consistent mean-variance risk sharing and reinsurance for an insurance group in a Stackelberg differential game. \emph{J. Ind. Manag. Optim.}, 21(1), 135-166, 2025.

\bibitem{Wonham68} W.M. Wonham, On a matrix Riccati equation of stochastic control. \emph{SIAM J. Control}, 6(4), 681-697, 1968.

\bibitem{WXZ24} F. Wu, J. Xiong, and X. Zhang, Zero-sum stochastic linear-quadratic Stackelberg differential games with jumps. \emph{Appl. Math. Optim.}, 89(1), 29, 2024.

\bibitem{Yong02} J.M. Yong, A leader-follower stochastic linear quadratic differential game. \emph{SIAM J. Control Optim.}, 41(4), 1015-1041, 2002.

\bibitem{ZS22} Y.Y. Zheng, J. T. Shi, A linear-quadratic partially observed Stackelberg stochastic differential game with application. \emph{Appl. Math. Comput.}, 420, 126819, 2022.

\bibitem{ZL00} X.Y. Zhou, D. Li, Continuous-time mean-variance portfolio selection: a stochastic LQ framework. \emph{Appl. Math. Optim.}, 42(1), 19-33, 2000.

\end{thebibliography}
\end{document}